\documentclass[10pt]{amsart}

\usepackage{amssymb}
\usepackage[all]{xy}
\usepackage{mathrsfs}
\usepackage{enumerate}

\usepackage{bbm}

\DeclareMathOperator{\Char}{{char}}

\newcommand{\rar}[1]{\stackrel{#1}{\longrightarrow}}

\newcommand{\bC}{{\mathbb C}}

\newcommand{\bG}{{\mathbb G}}

\newcommand{\bP}{{\mathbb P}}
\newcommand{\bQ}{{\mathbb Q}}

\newcommand{\bZ}{{\mathbb Z}}

\newcommand{\cB}{{\mathcal B}}
\newcommand{\cC}{{\mathcal C}}
\newcommand{\cD}{{\mathcal D}}
\newcommand{\cE}{{\mathcal E}}

\newcommand{\cG}{{\mathcal G}}
\newcommand{\cH}{{\mathcal H}}

\newcommand{\cM}{{\mathcal M}}

\newcommand{\cO}{{\mathcal O}}
\newcommand{\cP}{{\mathcal P}}
\newcommand{\cQ}{{\mathcal Q}}
\newcommand{\cR}{{\mathcal R}}
\newcommand{\cS}{{\mathcal S}}
\newcommand{\cT}{{\mathcal T}}
\newcommand{\cU}{{\mathcal U}}
\newcommand{\cV}{{\mathcal V}}

\newcommand{\ok}{{\overline k}}

\newcommand{\eff}{\text{eff}}
\newcommand{\et}{\text{et}}
\newcommand{\gm}{\text{gm}}
\newcommand{\Ho}{\text{Ho}}

\newcommand{\tri}{\text{tri}}
\newcommand{\trunc}{\text{trunc}}
\newcommand{\en}{\text{en}}
\newcommand{\gal}{\text{gal}}
\newcommand{\pretr}{\text{pretr}}
\newcommand{\w}{\text{w}}

\newcommand{\nc}{\newcommand}

\nc\wh{\widehat}

\nc\on{\operatorname}

\nc\Gr{\on{Gr}}

\nc\Fl{\on{Fl}}

\newcommand{\limto}{{\displaystyle\lim_{\longrightarrow}}}
\newcommand{\rightlim}{\mathop{\limto}}

\newcommand{\leftlim}{\mathop{\displaystyle\lim_{\longleftarrow}}}
\newcommand{\limfromn}{\leftlim\limits_{\raise3pt\hbox{$n$}}}
\newcommand{\limton}{\rightlim\limits_{\raise3pt\hbox{$n$}}}

\newcommand{\rightlimit}[1]{\mathop{\lim\limits_{\longrightarrow}}\limits%
                    _{\raise3pt\hbox{$\scriptstyle #1$}}}

\newcommand{\leftlimit}[1]{\mathop{\lim\limits_{\longleftarrow}}\limits%
                    _{\raise3pt\hbox{$\scriptstyle #1$}}}

\newcommand{\iso}{\buildrel{\sim}\over{\longrightarrow}}
\newcommand{\liso}{\buildrel{\sim}\over{\longleftarrow}}
\newcommand{\mono}{\hookrightarrow}

\DeclareMathOperator{\spec}{{spec}}

 \DeclareMathOperator{\Ext}{{Ext}}
\DeclareMathOperator{\Hom}{{Hom}}

\DeclareMathOperator{\Mor}{{Mor}}

\DeclareMathOperator{\LAlb}{{LAlb}}
\DeclareMathOperator{\T}{{T}}

\DeclareMathOperator{\Gal}{{Gal}}
\DeclareMathOperator{\cone}{{cone}}

\makeatletter

\newcommand{\Rmnum}[1]{\expandafter\@slowromancap\romannumeral #1@}
\makeatother

\newtheorem{Th}{Theorem}
\newtheorem{pr}{Proposition}[section]
\newtheorem{lm}[pr]{Lemma}

\theoremstyle{definition}

\newtheorem{rem}[pr]{Remark}

\numberwithin{equation}{section}

\newcommand{\tr}{\operatorname{tr}}

\begin{document}

\title{Some applications of weight structures of Bondarko}

\author{Vadim Vologodsky}

\address{Department of Mathematics, University of Oregon, Eugene, OR, 97403, USA}
\email{vvologod@uoregon.edu}



\keywords{Algebraic cycles, Voevodsky motives, Hodge theory.}

\subjclass[2000]{Primary 14F42, 14C25;  Secondary 14C22, 14F05.}


\begin{abstract}
We construct a functor from the triangulated category of Voevodsky motives to a certain derived category of mixed Hodge structures enriched with integral weight filtration.
We use this construction to prove a strong integral version of the Deligne conjecture on 1-motives which was known previously  only up to isogeny (\cite{bk}).   
\end{abstract}

\maketitle

\section{Introduction}
\subsection{Enriched Hodge structures.}\label{int.enr.hodge} Let $DM^{\eff}_{\gm}(\bC; \bZ)$ be the triangulated  category of Voevodsky's motives over $\bC$ with integral coefficients. One has the realization functor (\cite{hu1}, \cite{hu2}, \cite{dg}, \cite{vol})
\begin{equation}\label{hu}
R^{Hodge}_\bZ:  DM^{\eff}_{\gm}(\bC; \bZ) \to  D^b(MHS^\bZ )
\end{equation}
to the bounded derived category of polarizable mixed $\bZ$-Hodge structures. Recall that
a mixed $\bZ$-Hodge structure consists of a finitely generated abelian group $V_\bZ$ together with a weight filtration on the {\it rational} vector space $V_\bQ= V_\bZ \otimes \bQ$
and a Hodge filtration on the $\bC$-vector space $V_\bC$ satisfying certain compatibility conditions.  On the other hand, it has been shown by Gillet and Soul\'e (\cite{gs}) that the integral
homology of every complex algebraic variety carries a canonical {\it integral} weight filtration.  In  this paper we  construct  a realization functor
from the triangulated  category of Voevodsky's motives to a certain derived category of mixed Hodge structures with the weight filtration defined {\it integrally}. Our construction is based on the
notion of {\it  weight structure}  on a triangulated category developed by Bondarko (\cite{bon1}, \cite{bon3}).  The above realization functor is a key ingredient in our approach to the Deligne conjecture on 1-motives, which is the subject of the second part of this paper.

Let $A$ be a subring of $\bQ$.
We define an {\it enriched mixed $A$-Hodge structure}  
$V= (W_\cdot \subset V_A, F^{\cdot} \subset V_\bC )$ 
to be a finitely generated $A$-module $V_A$ together with a weight filtration on $V_A$  and a Hodge filtration on  $V_\bC$  such that
the triple  $(V_\bZ,  W_\cdot  \otimes \bQ, F^{\cdot} )$ constitutes a mixed Hodge structure in the sense of Deligne.
A morphism $f$ between enriched mixed $A$-Hodge structures is just a morphism between the underlying $A$-modules which is compatible  with the weight and the Hodge filtrations. We note that, in general, a morphism between enriched mixed $A$-Hodge structures is not {\it strictly} compatible with the weight filtration. As a result,  the category of  enriched mixed $A$-Hodge structure $MHS^{A, \en}$ is not abelian. However,  $MHS^{A, \en}$ 
carries a natural exact structure. We postulate that a sequence 
$$0\to V^0 \to V^1 \to V^2 \to 0$$ is exact if, 
for every $i$,  the sequence of pure integral Hodge structures   
$$0\to  \on Gr^i _W V^0 \to \on Gr^i _W V^1 \to \on Gr^i _W V^2\to 0$$ {\it splits}.
An enriched Hodge structure is called {\it effective} if $F^1=0$, $W_0=V_A$ and $\on Gr^0_W V_A$ is torsion free;
$V$ is called  polarizable if $(V_A,  W_\cdot  \otimes \bQ, F^{\cdot} )$ is polarizable. We denote the full subcategory of enriched effective polarizable mixed Hodge structures by 
$MHS^{A, \en}_{\eff}$. The exact structure on $MHS^{A, \en}$ induces an exact structure on  $MHS^{A, \en}_{\eff}$. 
  Let  $D^b(MHS_{\eff}^{A, \en})$ be the corresponding bounded derived category (\cite{n}).
 We lift the functor  $R^{Hodge}_A$ to a triangulated functor
  \begin{equation}\label{r.hodge}
 \hat R^{Hodge}_A:    DM^{\eff}_{\gm} (\bC;  A)     \to    D^b(MHS_{\eff}^{A, \en}).
   \end{equation}
      
 \subsection{Weight structures.}\label{int.weight.str} The main ingredients involved in the definition of (\ref{r.hodge}) are a canonical weight structure and a DG structure on the triangulated category of Voevodsky motives (\cite{bon2}, \cite{bv}) and the following
  general construction inspired by  (\cite{bon3}, \S 6). Let $\cC$ be a pretriangulated DG category\footnote{A ``DG glossary'' can be found in \S \ref{s.dgcat}.} together with a weight structure on the corresponding triangulated category $\cC^{\tri}$. Recall from \cite{bon3} that a {\it weight structure} on $\cC^{\tri}$ is just
  a class of objects $\cP\subset \cC^{\tri}$ closed under finite direct sums and direct summands such that  $\cC^{tri}$ is generated by $\cP$ and such that, for every $X, Y\in \cP$, we have\footnote{Though Bondarko's original definition of a (bounded) weight structure is different it is {\it a posteriori} equivalent to ours (\cite{bon3}, Theorem 4.3.2). The subcategory $\cP\subset \cC^{\tri}$ is the {\it heart}
  of a weight structure in the sense of  \cite{bon3}.}
      $$\Hom_{\cC^{\tri}}(X,Y[i])=0,  \, \text{for every} \, i>0.$$    
   Viewing $\cP$  as a full DG subcategory of $\cC$, we have that the complex
      $\Hom_\cC(X,Y)$ is acyclic in positive degrees. Let  $\cC_{\trunc} $ be the DG category having the same objects as  $\cC$ and whose complexes of morphisms $\Hom_{\cC_{\trunc}}(X,Y)$ are the canonical truncations of  $\Hom_{\cC}(X,Y)$ in degree $0$, and let $\cC^{\natural}$ be the pretriangulated completion of  $\cC_{\trunc}$. We have a canonical quasi-functor $can: \cC^{\natural} \to \cC$. In Proposition \ref{prweightmc} we explain how a weight structure $\cP$ on  $\cC^{\tri}$ determines  a section of $can$ 
      $$Q_\cP:  \cC \to \cC^{\natural}, \quad can\circ Q_\cP= Id.$$
    It follows that, for any pretriangulated DG category $\cC'$ and quasi-functor $R: \cC \to \cC'$, a weight structure on  $\cC^{\tri}$ yields 
    a commutative diagram
\begin{equation}\label{intro.mainweight}
\xymatrix{
\cC  \ar[dr]_{R}  \ar[r]^{\hat R} & \cC^{\prime \natural}  \ar[d]^{} \\
& \cC',\\
}
\end{equation}
where $\hat R$ is the composition
$$\hat R:  \cC \rar{Q_\cP} \cC^{\natural}\rar{R^{\natural}}\cC^{\prime \natural}.$$
We apply the above constriction to the DG category   $D\cM^{\eff}_{\gm} (k, A)$  of Voevodsky motives over a perfect field $k$ with coefficients in a commutative ring $A$ (\cite{bv}).  Its homotopy category $DM^{\eff}_{\gm} (k, A)$
is the triangulated category Voevodsky motives.
 Let $\cP \subset  DM^{\eff}_{\gm} (k;  A)$ consist of those motives that are isomorphic to direct summands of motives of smooth projective varieties.  It is proven in (\cite{bon2}) that if either $k$ is a field of characteristic $0$ or $\Char k$ 
  is invertible in $A$, then $\cP$ is a weight structure on $DM^{\eff}_{\gm} (k;  A)$.   In \cite{vol}, we constructed  a DG quasi-functor 
 $$D\cM^{\eff}_{\gm}(\bC; A) \to  D^b_{dg}(MHS^A_{\eff} ), \quad A\subset \bQ$$
 to the DG derived category of effective polarizable mixed $A$-Hodge structures which lifts (\ref{hu}). 
 Applying (\ref{intro.mainweight}) to this quasi-functor we get
 \begin{equation}\label{intro.mainconstrwe}
    D\cM^{\eff}_{\gm}(\bC; A) \to D^b_{dg}(MHS^A_{\eff} )^\natural.
   \end{equation} 
  If $\cB$  is an abelian category of homological dimension $\leq 1$  we identify  $D^b_{dg}(\cB )^\natural $ with the derived DG category $D^b_{dg}(F\cB )$ of the category\footnote{equipped with
  a certain peculiar exact structure} of filtered objects of $\cB$ (Prop. \ref{bisfiltr}).  Finally, we check that  (\ref{intro.mainconstrwe}) factors through the derived DG category of enriched Hodge structures $D^b_{dg}(MHS_{\eff}^{A, \en})$, which is a full subcategory
  of $D^b_{dg}(FMHS_{\eff}^{A})$. This gives (\ref{r.hodge}).

 \subsection{Motivic Albanese functor.}\label{int.mot.alb.} The derived category  $D^b(  \cM_1(k, \bQ))$ of 1-motives up to isogeny over a perfect field $k$ has been studied mainly by Barbieri-Viale and Kahn (\cite{bk}), with some contribution of Vologodsky  (\cite{vol}).
 One of the main results of this theory (\cite{bk}, Th. 6.2.1) asserts that the embedding
  $$D^b(\cM_1(k, \bQ))\mono DM^{\eff}_{\gm}( k ; \mathbb{Q})$$
admits a left adjoint
 $$\LAlb_\bQ:   DM^{\eff}_{\gm}( k ; \mathbb{Q}) \to D^b(\cM_1(k, \bQ)), $$
 called the {\it motivic Albanese functor}. 
 For $k\subset \bC$ the motivic Albanese functor fits into 
   the following
 commutative diagram (\cite{vol}, Th. 3; {\it cf.} \cite{bk}, Th. 17.3.1)
\begin{equation}\label{ratalbanese}
\def\normalbaselines{\baselineskip20pt
\lineskip3pt  \lineskiplimit3pt}
\def\mapright#1{\smash{
\mathop{\to}\limits^{#1}}}
\def\mapdown#1{\Big\downarrow\rlap
{$\vcenter{\hbox{$\scriptstyle#1$}}$}}
\begin{matrix}
DM^{\eff}_{\gm}( k ; \mathbb{Q})  & \rar{\LAlb_\bQ}  &D^b(\cM_1(k, \bQ))  \cr
  \mapdown{\hat R^{Hodge}_\bQ } &   &\mapdown{\T^{Hodge}_\bQ }  \cr
 D^b(MHS_{\eff}^{\bQ})   &\rar{\overline{\LAlb}_\bQ } &  D^b( MHS_1^{\bQ}),
\end{matrix}
 \end{equation}
 where $MHS_1^{\bQ}$ is the full subcategory of $MHS_{\eff}^{\bQ}$ consisting of mixed Hodge structures with
 possibly non-zero Hodge numbers in the set $\{(0,0),(0,-1),(-1, 0),(-1,-1)\}$, 
  $\T^{Hodge}_\bQ$ is the Hodge realization functor defined by Deligne (\cite{d3}), and the Hodge Albanese functor $\overline{\LAlb}_\bQ $ is left adjoint to the embedding
 $$D^b( MHS_1^{\bQ}) \mono  D^b(MHS_{\eff}^{\bQ}).$$
Our goal in this paper is to explain an integral version of the above results.  What are the main difficulties in the integral case? A first problem is that none of the categories $\cM_1(k)$, $MHS_1^{\bZ}$ is abelian\footnote{By definition, the category 
$MHS_1^{\bZ}$ consists of torsion free mixed $\bZ$-Hodge structures $V$ such that $V\otimes \bQ$ is in $MHS_1^{\bQ}$. The functor $\T^{Hodge}_\bZ:  \cM_1(\bC) \to MHS_1^{\bZ}$ is an equivalence of categories (\cite{d3}, \S 10.1).}. 
However,  there are natural exact structures one can put on these categories.  Given a subring $A\subset \bQ$ we view the category  $MHS_1^{A}$ as a full subcategory of the exact category $MHS_{\eff}^{A, \en}$ and consider the induced exact structure on  $MHS_1^{A}$.  
The corresponding derived category is denoted by $D^b(MHS_1^{A})$.
In Proposition \ref{hadjoint}
we prove that the embedding
 $$  D^b(MHS_1^{A}) \mono D^b(MHS_{\eff}^{A, \en})$$
admits a left adjoint $\overline{\LAlb}_A $.

Next, for any perfect field $k$, we consider an exact structure on  $\cM_1(k, A)$ postulating 
that a sequence  
\begin{equation}\label{intro.seq}
0\to M^0 \to M^1 \to M^2 \to 0
\end{equation}
is exact if  the sequence of pure 1-motives  over an algebraic closure $k\subset \ok$   
$$0\to  \on Gr^i _W f^*_\ok M^0 \to \on Gr^i _W f^*_\ok M^1 \to \on Gr^i _W f^*_\ok M^2\to 0, \quad i=0,-1, -2 $$ {\it splits}.
Here $f^*_\ok: \cM_1(k, A) \to \cM_1(\ok, A)$ is the pullback functor. If $k\subset \bC$ this condition is equivalent to the exactness of the sequence 
$$0\to \T^{Hodge}_A f^*_\bC M^0 \to \T^{Hodge}_A f^*_\bC M^1\to \T^{Hodge}_A f^*_\bC M^2\to 0$$
of enriched mixed Hodge structures.
In \S \ref{s.embedd} we construct a triangulated functor
$$\iota: D^b(\cM_1(k, A))\to DM^{\eff}_{\gal}(k;  A)$$ 
  from the derived category of 1-motives to the category of Galois motives $DM^{\eff}_{\gal}(k;  A)$, which is defined to be the quotient of Voevodsky's category $ DM^{\eff}(k;  A)$ by the full subcategory consisting of motives whose pullback
  in  $DM^{\eff}(\ok;  A)$ is $0$. $\Ext$ groups in the category  $DM^{\eff}_{\gal}(k;  A)$ can be expressed using Galois cohomology (see \S \ref{s.gm}).
  We do not know if $\iota$ is fully faithful, in general. Nevertheless,  
  in Theorem \ref{th2}, we prove if $k$  is  either a field of characteristic $0$ or $\Char k$ is invertible in $A$, then
  $\iota$ admits a left adjoint
    $$\LAlb_A:   DM^{\eff}_{\gm, \gal} (k;  A) \to D^b({\cM}_1(k; A)),$$
     $$\Hom( \LAlb_A(V), M)\iso \Hom(V, \iota(M)), \quad M\in D^b({\cM}_1(k; A)), \, V\in DM^{\eff}_{\gm, \gal}(k;  A),$$
    where $DM^{\eff}_{\gm, \gal}(k;  A)$ is the idempotent completion of the image of $DM^{\eff}_{\gm} (k;  A)$ in $DM^{\eff}_{\gal} (k;  A)$.
    More generally, for every perfect field $k$ and any $A$,  the motivic Albanese functor is defined  on the triangulated subcategory of $DM^{\eff}_{\gm, \gal}(k;  A)$ generated by Chow motives (Remark \ref{integralincharp}).

\begin{rem}\label{bkalb} The category $MHS_1^{A}$ has a stronger exact structure ({\it i.e.},  with more exact sequences) induced by its embedding into the  abelian category $MHS_{\eff}^{A}$. The corresponding
exact structure on the category of 1-motives can be defined geometrically, at least for every field $k$ of characteristic $0$ (\cite{bk}, \S 1.4).  The derived category of 1-motives equipped with this exact structure (which we denote, for the purpose of this paper, by $D^b(\cM_1^{\et}(k, A))$) has been studied in \cite{bk}. In particular, it is shown there how  $D^b(\cM_1^{\et}(k, A))$ can be realized as a full subcategory of  the category of  \'etale Voevodsky motives  $DM^{\eff}_{\gm, \et}( k ; A)$. Unfortunately, unless $A=\bQ$, the embedding $D^b(\cM_1^{\et}(k, A))\mono DM^{\eff}_{\gm, \et}( k ; A)$ does not admit a left adjoint (\cite{bk}, \S 5.2.2). However, in \cite{bk},
  Barbieri-Viale and Kahn using an {\it ad hoc} construction produced a functor from $ DM^{\eff}_{\gm} (k;  A)$ to the category $D^b(\cM_1^{\et}(k, A))$, which they also call the motivic Albanese functor.
    It can be shown that the Barbieri-Viale--Kahn functor is the composition of $\LAlb_A$ with the projection  $D^b(\cM_1(k, A)) \to D^b(\cM_1^{\et}(k, A))$.    
    \end{rem}
    
  Finally, for $k\subset \bC$ and $A\subset \bQ$,  we construct in Theorem \ref{mth} a commutative  
  diagram 
\begin{equation}\label{intro.delconj}
\def\normalbaselines{\baselineskip20pt
\lineskip3pt  \lineskiplimit3pt}
\def\mapright#1{\smash{
\mathop{\to}\limits^{#1}}}
\def\mapdown#1{\Big\downarrow\rlap
{$\vcenter{\hbox{$\scriptstyle#1$}}$}}
\begin{matrix}
DM^{\eff}_{\gm}( k ; A)  & \rar{\LAlb_A}  &D^b(\cM_1(k, A))  \cr
  \mapdown{\hat R^{Hodge}_A } &   &\mapdown{\T^{Hodge}_A }  \cr
 D^b(MHS_{\eff}^{A, \en})   &\rar{\overline{\LAlb}_A } &  D^b( MHS_1^{A}).
\end{matrix}
 \end{equation}

This assertion can be viewed as a {\it derived} version of the Deligne conjecture on 1-motives ({\it cf.} \cite{bk}, \S 17.3). Its relation to the ``classical'' Deligne conjecture is explained in Remark \ref{clasdelconj}.

 
 {\bf Acknowledgments.}  I would like to thank Leonid Positselski for helpful conversations related to the subject of this paper.  Special thanks go to the referees for careful reading the first draft of the paper
and for their numerous remarks and suggestions.  

\section{Some Homological Algebra.} \label{h.a}
\subsection{Generalities.}\label{s.dgcat} We begin by recalling some generalities on the {\it homotopy theory} of DG categories. The references are made to (\cite{t}, \cite{dri}, \cite{k}). 
Let $\cC$ be a DG category over a commutative ring $A$. Thus morphisms between objects of $\cC$ are complexes of $A$-modules
and composition maps are morphisms of complexes. 
The homotopy category $\Ho\, \cC$ has the same objects as $\cC$, and morphisms
$$\Hom_{\Ho\, \cC}(X,Y):= H^0\Hom_\cC(X,Y).$$
A DG functor  $F: \cC\to \cC'$  is called a homotopy equivalence (or quasi-equivalence) if for every $X,Y \in \cC$ the induced morphism of complexes $\Hom_\cC(X,Y)\to \Hom_\cC(F(X),F(Y))$
is a quasi-isomorphism and the functor $Ho\, F:   Ho\, \cC \to Ho\, \cC'$ is essentially surjective (and, therefore, by the first condition, an equivalence of categories).
The localization of the category of small DG categories with respect to quasi-equivalences is denoted by $Hqe$.  We refer
to $Hqe$ as {\it the homotopy category of small DG categories}. 

 The category $Hqe$ carries a symmetric monoidal structure 
 $$-\otimes^L -:  Hqe \times Hqe \to Hqe. $$
 defined as follows (\cite{t}, \S 4). If $\cC$  and $\cC'$ are DG categories and either $\cC$ or $\cC'$ is $A$-flat\footnote{Recall that a DG category $\cD$ is called $A$-flat if the functor $\Hom_\cD(X,Y) \otimes ?$ preserves quasi-isomorphisms for all $X,Y$ of $\cD$.} the derived tensor product $\cC\otimes ^L \cC'$ is just the usual tensor product $\cC\otimes  \cC'$ of DG categories {\it i.e.},  objects of $\cC\otimes  \cC'$ are pairs
 $(X, X')$, $X\in \cC$, $X' \in \cC'$,  and
 $$\Hom_{\cC\otimes  \cC'}((X, X'), (Y, Y')): = \Hom_\cC(X,Y)\otimes _A \Hom_{\cC}(X',Y').$$
 In general, for every DG category $\cC$ one finds an $A$-flat DG category $Q(\cC)$ with a quasi-equivalence $Q(\cC)\to \cC$ and sets
  $$\cC\otimes ^L \cC' : = Q(\cC)\otimes \cC'.$$
 In (\cite{t}, Theorem 6.1), Toen proved that  the symmetric monoidal  category $Hqe$ admits an internal Hom-functor, $R\cH om$: for small DG categories $\cC$ and $\cC'$ the DG category   $R\cH om(\cC, \cC')$, viewed as an object of $Hqe$, is
 characterized by the property that, for every DG category $\cD$,  one has a functorial isomorphism
 $$\Hom_{Hqe}(\cD,  R\cH om(\cC, \cC')) \iso \Hom_{Hqe}(\cD\otimes ^L \cC, \cC').$$ 
We write $\cT(\cC, \cC')$ for the homotopy category of $R\cH om(\cC, \cC')$. Objects of $\cT(\cC, \cC')$ are called {\it DG quasi-functors}.
By the universal property of $R\cH om$ every quasi-functor $F\in  \cT(\cC, \cC')$ induces a genuine functor $\Ho\, F:   \Ho\, \cC \to Ho\, \cC'$ between the homotopy categories.

Small DG categories and DG quasi-functors form a $2$-category denoted by $DGcat$ (\cite{dri}, \S 16).  We refer
to $DGcat$ as {\it the homotopy 2-category of small DG categories}.


Having the 2-category $DGcat$ we define the notion of {\it adjoint DG quasi-functors}:  given $F \in \cT(\cC, \cC')$  a right  {\it adjunction datum} $(G, \nu, \mu)$ consists of 
 a quasi-functor $G\in \cT(\cC', \cC)$ together with  morphisms
 $\nu: Id \to G \circ F,$ $\mu: F \circ G \simeq Id $ such that 
 the compositions
$$ F \rar{F(\nu)}
 F\circ G \circ F
 \rar{\mu(F)}     F$$
$$ G \rar{\nu (G)}  
 G \circ F \circ G
 \rar{G(\mu) } G $$
are identity morphisms.

For a full subcategory $\cB \subset \cC\in DGcat $ {\it the DG quotient} $\cC/\cB$ is a small DG category equipped with a DG quasi-functor $\cC \to  \cC/\cB$ satisfying the following universal property: for every DG category $\cC'\in DGcat  $ the functor
\begin{equation}\label{dgquotient}
\cT(\cC/\cB , \cC') \to \cT(\cC, \cC')
\end{equation}
 is fully faithful embedding whose essential image consists of quasi-functors $F \in   \cT(\cC, \cC')$ such that $\Ho(F)(\Ho(\cB))=0$.\footnote{In particular, if $\cB$ is nonempty and $\Ho(\cC')$ does not have a zero object   $\cT(\cC/\cB , \cC')$ is empty.}
 A DG quotient  $\cC/\cB$ always exists (and unique up to a unique isomorphism in $DGcat $).  
 In particular, given an $A$-linear exact category  $\cE$  we define its  bounded derived DG category  $D^{b}_{dg}({\cE})$ to be the DG quotient  of the DG category $C^b_{dg}({\cE})$ of bounded complexes by the subcategory of acyclic ones (\cite{n}, \S 1). The homotopy category of $D^{b}_{dg}({\cE})$ is the usual bounded derived category $D^{b}({\cE})$. We will write  $D_{dg}({\cE})$ unbounded derived DG category of $\cE$.

 Given a small DG category  $\cC$ one defines its  {\it DG proj-completion} $\underleftarrow{\cC}$ to be 
 $$\underleftarrow{\cC}= R\cH om(\cC, D_{dg}(A-mod))^{op},$$
 where superscript $op$ stands for the opposite category\footnote{To avoid set-theoretical problems one should fix a universe $\cU$ such that
 $\cC$ is $\cU$-small and let  $D_{dg}(A-mod)$ be the derived DG category of $\cU$-small $A$-modules. Then
 $\underleftarrow{\cC}$ is $\cV$-small for some $\cU\in \cV$ ({\it cf.} \cite{t}, \S 3).}. 
 The dual notion, the {\it DG ind-completion}  $\underrightarrow{\cC}$ of $\cC$, is defined by 
  $$\underrightarrow{\cC}= (\underleftarrow{\cC ^{op}})^{op}  = R\cH om(\cC^{op}, D_{dg}(A-mod)).$$
 One has the Yoneda  quasi-functors 
 $$\cC\mono  \underleftarrow{\cC}, \quad \cC \mono \underrightarrow{\cC}$$
  that carry an object $X\in \cC$ to the DG functors $\Hom_\cC(X, ?)$ and $\Hom_\cC(?, X)$ respectively.  
  
   We refer the reader to (\cite{dri}, \S 14 or \cite{bv} \S 1.6) for a more explicit construction of $\underrightarrow{\cC}$ and $\underleftarrow{\cC}$. Let us just mention that another name for $(\underleftarrow{\cC})^{op}$  (resp.  $\underrightarrow{\cC}$)
 used in the literature is the {\it  derived DG category of left (resp. right) DG modules over $\cC$.}  
 Given a quasi-functor $F: \cC \to \cC'$ the restriction quasi-functor $F^*: \underleftarrow{\cC'}\to \underleftarrow{\cC}$
admits right and left adjoint ones, the derived induction and co-induction functors, (\cite{dri}, \S 14.12).
\begin{equation}\label{ad}
F_*, F_! :  \underleftarrow{\cC} \to \underleftarrow{\cC'}.
\end{equation}


For $X\in \cC$ its $n$-translation , $n\in \bZ$,  is an object  $X[n]$ of $\cC$ that represents the DG functor $Y \mapsto \Hom_\cC(Y, X)[n]$.
Similarly, for a closed morphisms $\phi: X\to X'$ of degree $0$, the cone of $\phi$ is an object $\cone(\phi)$ that represents the DG functor
 $$Y\mapsto \cone(\Hom_\cC(Y, X)\rar{\phi_*} \Hom_\cC(Y, X')).$$
We shall say that $\cC$ is a {\it strongly pretriangulated} DG category if $\cC$ has $0$ object and  $X[n]$, $\cone(\phi)$ exist for every 
$X$, $n$, and $\phi$. If $\cC$ is a strongly pretriangulated DG category its homotopy category $\Ho\, \cC$ has a natural structure of triangulated category and every DG functor $F: \cC\to \cC'$ between strongly pretriangulated DG categories induced a triangulated functor between the 
corresponding homotopy categories. 
For every DG category category $\cC$ its {\it pretriangulated completion}
is a strongly pretriangulated DG category
$\cC^{\pretr}$  together with a DG functor $\cC\to \cC^{\pretr}$ satisfying the following universal property: every DG functor $F: \cC\to \cC'$ 
to a strongly pretriangulated DG category factors uniquely through $\cC^{\pretr}$.  Every DG category has a pretriangulated completion
and we refer the reader to the original paper by  Bondal and Kapranov (\cite{boka}) or to (\cite{dri}, \S 2.4) for an explicit construction of $\cC^{\pretr}$.
Informally, $\cC^{\pretr}$ is obtained from
$\cC$   by adding to $\cC$ all cones, cones of morphisms between cones, etc.  For example, the pretriangulated completion
of an additive category $\cB$ viewed as a DG category is just the category of bounded complexes $C^b(\cB)$.
We shall say that $\cC$ is a {\it pretriangulated} DG category if $\cC \to \cC^{\pretr}$ is a homotopy equivalence. For a pretriangulated DG category $\cC'$ and any category $\cC$ the restriction functor
$$\cT(\cC^{\pretr} , \cC') \to \cT(\cC, \cC')$$
is an equivalence of categories.

\subsection{Truncation of a DG category.}\label{s.trunc}
Given a DG category $\cC$ we denote by $\cC_{\trunc} $ the DG category having the same objects as  $\cC$ and morphisms defined by the formula
$$Hom_{\cC_{\trunc}}(X,Y)= \tau_{\leq 0}\Hom_{\cC}(X,Y),$$
where $\tau_{\leq i}C^{\cdot}$ denotes the canonical truncation of a complex $C^{\cdot}$ in degree $i$.

We say that $\cC$ is {\it negative}  if for every objects $X,Y\in \cC$ the complex $\Hom_\cC(X,Y)$ is acyclic in positive degrees. 
\begin{pr}\label{propoftrunc} Let $\cC$ be a small negative DG category. 
Then, for every small DG category $\cC'$, the functor $\cC'_{\trunc}\to \cC'$ induces an equivalence of categories
$$ \cT(\cC, \cC'_{\trunc})\iso  \cT(\cC, \cC').$$
\end{pr}
\begin{proof}
We shall prove a stronger statement that the functor
\begin{equation}\label{eq57}
R\cH om(\cC, \cC'_{\trunc})_{\trunc} \to R\cH om(\cC, \cC')_{\trunc} 
\end{equation}
is an isomorphism in $Hqe$. Let us first check that the map
\begin{equation}\label{eq56}
\Hom_{Hqe}(\cC, \cC'_{\trunc})\to \Hom_{Hqe}(\cC, \cC')
\end{equation}
is a bijection. Indeed, since by our assumption the functor $\cC_{\trunc}\to \cC$ is a homotopy equivalence we may assume that the complexes  $\Hom_\cC(X,Y)$
are supported in non-positive degrees for all $X,Y$ of $\cC$. 
Moreover, using a construction from (\cite{dri}, Lemma 13.5) we may replace $\cC$ by a semi-free (and, therefore, cofibrant)
 DG category having the same property.
Then the sets  $\Hom_{Hqe}(\cC, \cC'_{\trunc})$ and $\Hom_{Hqe}(\cC, \cC')$ are in bijection with the sets of DG functors $\Hom (\cC, \cC'_{\trunc})$ and $\Hom(\cC, \cC')$
  modulo the homotopy relation (\cite{k}, \S 4.2). Finally, our assumption on $\cC$ implies that every DG functor $\cC\to \cC'$ factors uniquely through $\cC'_{\trunc}$.
  This proves (\ref{eq56}). To prove (\ref{eq57}) it is enough to check that for every $A$-flat DG category $\cD$ such that   $\cD_{\trunc}= \cD$ the map
$$\Hom_{Hqe}(\cD,  R\cH om(\cC, \cC'_{\trunc})_{\trunc}) \to \Hom_{Hqe}(\cD, R\cH om(\cC, \cC')_{\trunc})$$
is a bijection.  Using (\ref{eq56}) we reduce the problem to showing that the map
$$\Hom_{Hqe}(\cD,  R\cH om(\cC, \cC'_{\trunc})) \to \Hom_{Hqe}(\cD, R\cH om(\cC, \cC'))$$
is a bijection or, equivalently, that the map
$$\Hom_{Hqe}(\cD \otimes _A \cC, \cC'_{\trunc} ) \to \Hom_{Hqe}(\cD \otimes _A \cC, \cC' )$$
 has the same property, which is indeed the case by (\ref{eq56}).
\end{proof}
Using the proposition we define a functor
\begin{equation}\label{funct1}
 \cT(\cC, \cC')\iso  \cT(\cC_{\trunc}, \cC'_{\trunc})
\end{equation}
to be the composition of the restriction functor $\cT(\cC, \cC')\to \cT(\cC_{\trunc}, \cC')$ and the inverse to the equivalence $\cT(\cC_{\trunc}, \cC'_{\trunc}) \iso \cT(\cC_{trunc}, \cC')$.

The DG category $\cC_{\trunc}$ is not pretriangulated even if $\cC$ is. We denote by 
$\cC^\natural $ the pretriangulated completion of  $\cC_{\trunc}$. 
The formation of $\cC^\natural $ respects DG quasi-functors:  composing (\ref{funct1}) with $\cT(\cC_{\trunc}, \cC'_{\trunc})\to  \cT(\cC^\natural, \cC^{\prime \natural})$ we get a functor
\begin{equation}\label{funct2}
 \cT(\cC, \cC')\to  \cT(\cC^\natural, \cC^{\prime \natural}).
\end{equation}
Given a quasi-functor $R\in \cT(\cC, \cC')$ we shall denote its image in $\cT(\cC^\natural, \cC^{\prime \natural})$ by $R^\natural $.
\subsection{Weight structures.}\label{s.w.struc} This notion has been introduced by Bondarko (\cite{bon1}). We recall here one of the equivalent definitions of a weight structure given in (\cite{bon3}).   
 Let $\cC$ be a pretriangulated DG category, $\cC^{\tri}=\Ho\, \cC$ its homotopy category, and let $\cP\subset \cC$ be a full DG subcategory of $\cC$ (not necessary pretriangulated).   
 We say that $\cC$ is generated by $\cP$ if $\cC^{\tri}$ is the smallest triangulated subcategory of $\cC^{\tri}$ which contains $\Ho\, \cP$, or equivalently, if the DG quasi-functor $\cP^{\pretr}\to \cC$ is a quasi-equivalence.
  The subcategory $\cP$ is called homotopy idempotent complete in $\cC$ if any object of $\cC$ which is isomorphic in $\cC^{\tri}$
 to a direct summand of an object of $\Ho\, \cP $ lies in $\cP$. We say that $\cP$ is homotopy additive if $\Ho(\cP)$ is additive.
 
 A {\it weight structure}\footnote{Warning: in this paper we use the term  ``a weight structure'' for what is called ``a bounded weight structure''  in Bondarko's works (\cite{bon3}, Th. 4.3.2).} 
 on $\cC$ is a full DG subcategory $\cP\subset \cC$, which is negative ({\it i.e.}, for every objects $X,Y\in \cP$ the complex $\Hom_\cC(X,Y)$ is acyclic in positive degrees), homotopy idempotent complete, homotopy additive, 
 and which generates $\cC$.
 
 If $\cP'\subset \cC$ is a full DG subcategory  which is negative and which generates $\cC$ then there is a unique weight structure $\cP$ with  $\cP'\subset \cP$ (\cite{bon3}, Th. 4.3.2).
 
If $\cP$ is a weight structure on $\cC$ then, using Proposition \ref{propoftrunc},
the embedding $\cP\mono \cC$ lifts canonically to a DG quasi-functor $\cP\mono \cC_{\trunc}$, which, in turn,  extends to a quasi-functor $\cC\simeq \cP^{\pretr}\to \cC^\natural$.
We summarize this in the following proposition. 
\begin{pr}\label{prweightmc} A  weight structure $\cP$ on $\cC$ 
  determines  a DG quasi-functor $Q_\cP$ fitting into the commutative diagram
  \begin{equation}\label{weightdia1}
\xymatrix{
\cC  \ar[dr]_{Id}  \ar[r]^{Q_\cP} & \cC^{\natural}  \ar[d]^{can} \\
& \cC,\\
}
\end{equation}
 where $can:  \cC^{\natural}  \to \cC$ is induced by the functor $ \cC_{\trunc}  \to \cC$.
  \end{pr}     
 A weight structure on $\cC$ also determines a $t$-structure on the derived category $\cT(\cC, D_{dg}(A-mod))$ of DG modules over $\cC$.
 More generally, we have the following result.
  \begin{pr}\label{tstructure}
   Let $\cC$ be a DG category equipped with a weight structure $\cP\subset \cC$, $\cC'$ a pretriangulated DG category together with a $t$-structure $\Ho\, \cC^{\prime \leq 0}$, $\Ho \, \cC^{\prime \geq 0}$ on the
  homotopy category of $\cC'$, and let $\cT(\cC, \cC')^{\leq 0}$ (resp. $\cT(\cC, \cC')^{\geq 0}$) be the full subcategory of the triangulated category of DG quasi-functors $\cT(\cC, \cC')$ consisting of those $F\in  \cT(\cC, \cC')$ which carries
  every object of $\Ho\, \cP$ to an object of  $\Ho\, \cC^{\prime \leq 0}$ (resp. $\Ho\, \cC^{\prime \geq 0}$). Then, the subcategories $\cT(\cC, \cC')^{\leq 0}, \cT(\cC, \cC')^{\geq 0} \subset \cT(\cC, \cC')$ constitute a $t$-structure on   $\cT(\cC, \cC')$.
   \end{pr}   
 \begin{proof}
 We just explain a construction of the functor
 $$\tau_{\leq 0}: \cT(\cC, \cC') \to \cT(\cC, \cC')^{\leq 0}$$
 right adjoint to the embedding $\cT(\cC, \cC')^{\leq 0}\mono \cT(\cC, \cC')$.
 Let $\cC^{\prime \leq 0}$ be the full subcategory consisting of those objects which are isomorphic in $\Ho\, C'$ to an object of $\Ho\,  \cC^{\prime \leq 0}$. The embedding $(\cC^{\prime \leq 0})_{\trunc} \to \cC'_{\trunc}$ admits a right adjoint DG quasi-functor
  $\cC'_{\trunc} \to  (\cC^{\prime \leq 0})_{\trunc}$
 that carries each object $X\in \cC'_{\trunc}$
to an object quasi-isomorphic to $\tau_{\leq 0}X$. 
Using Proposition \ref{propoftrunc} this yields a functor
$$\cT(\cP, \cC') \iso \cT(\cP, \cC'_{\trunc})\to  \cT(\cP, (\cC^{\prime \leq 0})_{\trunc}) \iso \cT(\cP, \cC^{\prime \leq 0})$$
 right adjoint to $\cT(\cP, \cC^{\prime \leq 0})\mono \cT(\cP, \cC')$. Finally, since $\cP^{\pretr}\iso \cC$, we have that
 $$\cT(\cC, \cC')  \iso \cT(\cP, \cC'), \quad  \cT(\cC, \cC')^{\leq 0}\iso  \cT(\cP, \cC^{\prime \leq 0}).$$
 \end{proof} 
 \begin{rem}\label{tstructurebis} The above proof also shows that, for every negative DG category $\cP$ and a pretriangulated DG category $\cC'$ with a $t$-structure on $\Ho\, \cC'$, the subcategories $\cT(\cP,   \cC^{\prime \leq 0}),  \cT(\cP,   \cC^{\prime \geq 0}) \subset \cT(\cP, \cC')$ constitute a $t$-structure on   $\cT(\cP, \cC')$.
 \end{rem} 
\subsection{The Chow weight structure on the category of Voevodsky motives.}\label{chow}  This is the most important for us example of a weight structure. 
 Denote by $D\cM^{\eff}_{\gm} (k, A)$  the DG category of effective Voevodsky motives over a perfect field $k$ with coefficients in a commutative ring $A$ (\cite{bv}, \S 2).
 Its homotopy category $DM^{\eff}_{\gm} (k, A)$
is the triangulated category of effective Voevodsky motives.
 Let $\cP \subset  D\cM^{\eff}_{\gm} (k;  A)$ be the full
 DG subcategory consisting of those motives that are isomorphic in $DM^{\eff}_{\gm} (k;  A)$ to direct summands of motives of smooth projective varieties. 
  It is proven in (\cite{bon2}, Th. 2.1.1) that if either $\Char k=0$  or  $\Char k$ is invertible in $A$, then $\cP$ is a weight structure on $D\cM^{\eff}_{\gm} (k;  A)$.
 Note that the homotopy category $\Ho\, \cP$ is equivalent to the category of pure effective Chow motives (\cite{voe}, Cor. 4.2.6).
  
  Given any pretriangulated DG category $\cC$ and quasi-functor $R: D\cM^{\eff}_{\gm} (k;  A) \to \cC$ we define a quasi-functor 
  $$\hat R: D\cM^{\eff}_{\gm} (k;  A) \to \cC^\natural$$
  to be the composition 
  $$\hat R:  D\cM^{\eff}_{\gm} (k;  A)\rar{Q_\cP} D\cM^{\eff}_{\gm} (k;  A)^\natural \rar{R^\natural } \cC^\natural .$$
  The functor $\hat R$ fits into the following commutative diagram
  \begin{equation}\label{mainweight}
\xymatrix{
D\cM^{\eff}_{\gm} (k;  A)  \ar[dr]_{R}  \ar[r]^{\hat R} & \cC^{ \natural}  \ar[d]^{can} \\
& \cC.\\
}
\end{equation}
\subsection{Some computations}\label{s.f.d.c} 
 The principal result of this subsection is an explicit description of the category $D^b_{dg}(\cB)^\natural$, where $\cB$ is an abelian category of homological dimension $\leq 1$.

 Let $A$ be a commutative ring, $\cB$ an $A$-linear abelian category, and let $F\cB$ be the category of finitely filtered objects of $\cB$.  
That is, objects of $F\cB$ are pairs $(V, W_{\cdot}V)$, where  $V\in \cB$ and $W_{\cdot} V$ 
is a increasing filtration on $V$,
such that, for sufficiently large $i$, we have $W_iV=V$, $W_{-i}V=0$. Morphisms between  $(V, W_{\cdot} V)$ and $(V', W_{\cdot} V')$ form a subgroup of $\Hom_{\cB}(V, V')$ consisting of those $f: V\to V'$ that carry $W_iV$ to $W_iV'$, for every $i$. 
We introduce an exact structure on $F\cB$ postulating that a sequence 
$$0\to (V^0, W_{\cdot} V^0) \to (V^1, W_{\cdot} V^1) \to (V^2, W_{\cdot} V^2) \to 0$$ is exact if, 
for every $i$,  the sequence    
$$0\to  \on Gr^i _W V^0 \to \on Gr^i _W V^1 \to \on Gr^i _W V^2\to 0$$ splits. We denote by $D^b_{dg}(F\cB)$ the corresponding derived DG category over $A$ and by $D^b(F\cB)$ the triangulated derived category (\cite{n}).
For an object $V$ of $\cB$, let  $V\{i\}$ be the object of $F\cB$ with $W_jV\{i\}$ equal to $V$ if $j\geq - i$ and to $0$ otherwise. The following result is due to Positselski (\cite{p}, Example in \S 8).
\begin{lm} \label{posiclm} If  $n\leq j-i$ the forgetful functor $\Phi: D^b(F\cB)\to D^b(\cB)$ induces an isomorphism 
$$
\Hom_{D^b(F\cB)}(V\{i\}, V'\{j\}[n])\iso \Hom_{D^b(\cB)}(V, V'[n]).
$$
If $n> j-i$  then $\Hom_{D^b(F\cB)}(V\{i\}, V'\{j\}[n])$ is trivial.
\end{lm}
\begin{proof} Since a proof is only sketched in (\cite{p}) we give a complete argument here. 
First, we check using induction on $n$ that $\Ext^n_{F\cB}(V\{i\}, V'\{j\})=0$ is for  $n> j-i$. The group $\Ext^1_{F\cB}(V\{i\}, V'\{j\})$ classifies extensions $0\to V'\{j\} \to X \to V\{i\}\to 0$ in $F\cB$. If $j\leq i$ every such extension splits: for $i=j$ this follows from the definition of the exact structure on $F\cB$ and for $j< i$ the splitting is given by $W_{-i}X\iso V\{i\}$. Given $n>1$ and  $\alpha \in \Ext^n_{F\cB}(V\{i\}, V'\{j\})$
we can find $X\in F\cB$ and  $\alpha_{n-1} \in \Ext^{n-1}_{F\cB}(V\{i\}, X)$, $\alpha_1 \in \Ext^{1}_{F\cB}(X, V'\{j\})$ such that $\alpha_{1} \alpha_{n-1}=\alpha$. Using the induction assumption
$\alpha_{n-1}: V\{i\}\to X[n-1]$ factors through $W_{-i-n+1}X[n-1]\to X[n-1]$. Thus, we may assume that $W_{-i-n+1}X=X$. But then $\Ext^{1}_{F\cB}(X, V'\{j\})=0$ since $n> j-i$. 

We now prove that for every $X, Y\in F\cB$ with $W_{-i-1}X=0$ and $W_{-j}Y=Y$ the morphism
  \begin{equation}\label{eq30}
 \Ext^n_{F\cB}(X, Y)\to \Ext^n_{\cB}(\Phi(X), \Phi(Y))
 \end{equation}
 is an isomorphism provided that   $n\leq j-i$ (this is a generalization of the assertion in the Lemma). For surjectivity, given $\alpha \in \Ext^n_{\cB}(\Phi(X), \Phi(Y))$ we can find objects $\Phi(X)=V_0, V_1, \cdots V_n=\Phi(Y)$ of $\cB$
 and elements $\alpha_m\in  \Ext^1_\cB(V_m, V_{m+1})$, $m=0,\cdots n-1$, such that $\prod \alpha_m= \alpha$. 
 Pick integers $i=l_0<l_1<\cdots < l_n=j$ (this is possible
because  $n\leq j-i$). Then the morphisms
$$\Ext^1_{F\cB}(X, V_{1}\{l_{1}\})\to \Ext^1_{\cB}(\Phi(X), V_{1}),$$
$$ \Ext^1_{F\cB}(V_m\{l_m\}, V_{m+1}\{l_{m+1}\})\to \Ext^1_{\cB}(V_m, V_{m+1}), \quad 1\leq m \leq n-2,$$
$$ \Ext^1_{F\cB}(V_{n-1}\{l_{n-1}\}, X)\to \Ext^1_{\cB}(V_{n-1}, \Phi(Y))$$
are isomorphisms. Let $\tilde \alpha_m$, $m=0,\cdots n-1$, be the preimages of $\alpha_m$ under the above isomorphisms. 
Then the image of $\prod \tilde \alpha_m \in \Ext^n_{F\cB}(X, Y)$ in $\Ext_\cB^n(\Phi(X), \Phi(Y))$ is our $\alpha$.
 
It remains to prove injectivity of (\ref{eq30}). We do this by induction on $n$. If $n=1$ the statement is clear. For the induction step it is enough to check injectivity of (\ref{eq30}) for
$X=V\{i\}$ and  $Y=V'\{j\}$ with $n\leq j-i$.
  Let  $\alpha \in \Ext^n_{F\cB}(V\{i\}, V'\{j\})$ be an element whose image in $\Ext^n_{\cB}(V, V')$ is zero. As at the first step of the proof
we can find $X=W_{-i-n+1}X \in F\cB$ and  $\alpha_{n-1} \in \Ext^{n-1}_{F\cB}(V\{i\}, X)$, $\alpha_1 \in \Ext^{1}_{F\cB}(X, V'\{j\})$ such that $\alpha_{1}\alpha_{n-1}=\alpha$. Let 
$$0\to V'\{j\}\rar{ } Y\rar{h} X\to 0$$
be a short exact sequence in $F\cB$ that represents $\alpha_1$.   Consider the following commutative diagram
   \begin{equation}
\def\normalbaselines{\baselineskip20pt
\lineskip3pt  \lineskiplimit3pt}
\def\mapright#1{\smash{
\mathop{\to}\limits^{#1}}}
\def\mapdown#1{\Big\downarrow\rlap
{$\vcenter{\hbox{$\scriptstyle#1$}}$}}
\begin{matrix}
\Ext^{n-1}_{F\cB}(V\{i\}, Y)       &\rar{h_*} &    \Ext^{n-1}_{F\cB}(V\{i\}, X)       &\rar{\cdot \alpha_1}&          \Ext^n_{F\cB}(V\{i\}, V'\{j\}) \cr
   \mapdown{\simeq}  & & \mapdown{\simeq}   & & \mapdown{} \cr
 \Ext^{n-1}_{\cB}(V, \Phi(Y))      &  \rar{}  & \Ext^{n-1}_{\cB}(V, \Phi(X)) &\rar{ }&          \Ext^n_{\cB}(V, V'), 
\end{matrix}
 \end{equation}
We observe that $\alpha=0$ if and  only if $\alpha_{n-1}$ lies in the image of $h_*$. An easy diagram chase using the  exactness of the rows completes the proof. 


\end{proof}

\begin{pr}\label{bisfiltr}  Let $A$ be a commutative ring, $\cB$ a small $A$-linear abelian category of homological dimension $\leq 1$.  Then the category $D^b_{dg}(\cB)^\natural$ is quasi-equivalent to $D^b_{dg}(F\cB)$. 
\end{pr}
\begin{proof}
Let $\cQ$ be the full DG subcategory of  $D^b_{dg}(F\cB)$ formed by objects $V\{i\}[i]$, with $V\in \cB$, $i\in \bZ$ . By Lemma \ref{posiclm}
the category $\cQ$ is negative. Hence, using Proposition \ref{propoftrunc}, 
the forgetful functor $\cQ \mono D^b_{dg}(F\cB)\to D^b_{dg}(\cB)$ factors as follows
$$\cQ \to D^b_{dg}(\cB)_{\trunc} \to D^b_{dg}(\cB).$$
Since $D^b_{dg}(F\cB)$ is generated by $\cQ$,  $\cQ^{\pretr}\iso D^b_{dg}(F\cB)$,  we get a quasi-functor 
 \begin{equation}\label{posic}
D^b_{dg}(F\cB)\to  (D^b_{dg}(\cB)^\natural
 \end{equation}
 which is a homotopically fully faithful by the Lemma \ref{posiclm}.  As $\cB$ has homological dimension $\leq 1$,  every complex in $D^b_{dg}(\cB)_{\trunc}$ is homotopy equivalent to a finite direct sum of objects in the image of $\cQ$.  It follows
 that (\ref{posic}) is a homotopy equivalence.
\end{proof}

\section{Enriched Hodge structures.} \label{e.h.s.}
\subsection{Generalities.}\label{s.g.c}
We start with a general construction. Let $k$ be a perfect field, $A$ a commutative ring,  and let $\cB$ be an $A$-linear abelian category of homological dimension $\leq 1$. 
We assume that either $\Char k=0$ or that $\Char k$ is invertible in $A$.
Let  $D\cM^{\eff}_{\gm} (k;  A)$
 be the DG category of effective geometric Voevodsky motives (\cite{bv}, \S 2), and let 
$$R: D\cM^{\eff}_{\gm} (k;  A) \to D^b_{dg}(\cB)$$
be a DG quasi-functor. Using diagram (\ref{mainweight}) and Proposition \ref{bisfiltr} it follows that the Chow weight structure on $D\cM^{\eff}_{\gm} (k;  A)$ determines a quasi-functor $\hat R$ fitting into the commutative diagram
  \begin{equation}\label{bondarko}
\xymatrix{
D\cM^{\eff}_{\gm} (k;  A)  \ar[dr]_{R}  \ar[r]^{\hat R} & D^b_{dg}(F\cB)  \ar[d]^{\Phi} \\
& D^b_{dg}(\cB).\\
}
\end{equation}
Here $\Phi$ is induced by the forgetful functor $F\cB \to \cB$.  

  Let us record some elementary properties of $\hat R$, which follow directly from the construction.
   \begin{enumerate}[(i)] 
\item For an integer $i$, let
$$\on Gr^i_W: D^b_{dg}(F\cB)\to C^b(\cB)$$
be the quasi-functor to the DG category of chain complexes over $\cB$,
 which carries a complex $(V^\cdot, W_\cdot  V^\cdot)\in D^b_{dg}(F\cB)$ to $\on Gr^i_W V^\cdot$. Then the quasi-functor 
 $$\on Gr^i_W \circ \hat R:  D\cM^{\eff}_{\gm} (k;  A)  \to C^b(\cB)$$
 carries every object $X$ of $\cP$ to the one-term complex 
  \begin{equation}\label{smproj}
 \on Gr^i_W(\hat R (X))\simeq H^i R(X)[-i].
 \end{equation}
 \item Recall from Proposition \ref{tstructure} that the weight structure $\cP\subset  D\cM^{\eff}_{\gm} (k;  A)$ determines a $t$-structure on the triangulated category $\cT(D\cM^{\eff}_{\gm} (k;  A), D^b_{dg}(\cB))$ of quasi-functors. Let
 $\tau_{\leq i} R$ be the canonical truncation of $R$ relative to this $t$-structure. Then   $\tau_{\leq i} R$ is isomorphic to 
 the composition
 $$ D\cM^{\eff}_{\gm} (k;  A)  \rar{\hat R} D^b_{dg}(F\cB)\rar{W_i}  D^b_{dg}(\cB),$$
 which carries $M\in D\cM^{\eff}_{\gm} (k;  A) $ to $W_i \hat R(M)$.  
 \item For every $X, Y\in \cP$ the forgetful functor $\Phi$ yields a quasi-isomorphism
    \begin{equation}\label{smproj2}
\Hom_{D^b_{dg}(F\cB)}(\hat R(X), \hat R(Y))\iso  \tau_{\leq 0}\Hom_{D^b_{dg}(\cB)}(R(X), R(Y)).
\end{equation}
\end{enumerate} 
  
 \subsection{Enriched Hodge realization.}\label{s.in.h.str.} Let  $A$ denote a subring of $\bQ$. 
We apply the above construction to the Hodge realization DG quasi-functor 
\begin{equation}\label{hodgemain}
 R^{Hodge}_A:  D\cM^{\eff}_{\gm} (\mathbb{C};  A)\to  D_{dg}^b(MHS^{A}).
 \end{equation}  
constructed in (\cite{vol}, \S 2).  By Lemma 1.1 from \cite{bei} the category $MHS^{A}$ of  polarizable mixed 
$A$-Hodge structures has cohomological dimension $1$. Consider the full subcategory $MHS^{A, \en}$  of $FMHS^{A}$ 
  whose objects (called enriched mixed Hodge structures) are polarizable mixed $A$-Hodge structures $V$ equipped with a filtration $W_{\cdot}V$ such that $\on Gr^i_W V$ is pure of weight $i$ (in particular, the filtration $W_{\cdot}V\otimes_A \bQ$ coincides with the weight filtration on $V\otimes _A \bQ$).
  The category $MHS^{A, \en}$ inherits  an exact structure from  $FMHS^{A}$.  An enriched mixed Hodge structure $(V, W_{\cdot}) \in MHS^{A, \en}$ is called effective if $F^1=0$, $W_0V=V$ and $\on Gr^0_W V$ is torsion free. 
  Denote by $MHS^{A, \en}_{\eff}\subset MHS^{A, \en}$
   the full subcategory of effective enriched mixed Hodge structures. 
   \begin{lm}  The functor $D^b(MHS^{A, \en}_{\eff}) \to D^b( FMHS^{A})$ is a fully faithful embedding.   
\end{lm}
\begin{proof} As $MHS^{A}$ has homological dimension $1$,  Lemma \ref{posiclm} implies that, for every objects $V, V' \in FMHS^{A}$ and every $i>1$, we have
$$Hom_{FMHS^{A}}(V, V'[i])=0.$$
Since the subcategory  $MHS^{A, \en}_{\eff} \to FMHS^{A}$ is closed under extensions the lemma follows from (\cite{p}, Prop. A.7).
\end{proof}
The lemma together with formula (\ref{smproj}) imply that the functor $\hat R^{Hodge}_A$  from the diagram  (\ref{bondarko}) factors through $D_{dg}^b(MHS^{A, \en}_{\eff})$ yielding
the enriched Hodge realization functor
 $$  D\cM^{\eff}_{\gm} (\bC;  A) \to  D_{dg}^b(MHS^{A, \en}_{\eff}).$$
  {\bf Notation:}  If $k$ is a subfield of $\bC$,  we will write, by abuse of notation,  $\hat R^{Hodge}_A$ for the composition
\begin{equation}\label{hodgemaineff}
\hat R^{Hodge}_A:  D\cM^{\eff}_{\gm} (k;  A) \rar{f^*} D\cM^{\eff}_{\gm} (\mathbb{C};  A)\to  D_{dg}^b(MHS^{A, \en}_{\eff}) 
 \end{equation}
where $f^*$ is the base change functor.
 
 \begin{rem} Our category  $MHS^{\bQ, \en}$ is just $MHS^{\bQ}$ and   $\hat R^{Hodge}_\bQ=R^{Hodge}_\bQ$. On the other hand,  the integral realization functor $\hat R^{Hodge}_\bZ$, in contrast with $R^{Hodge}_\bZ$, does {\it not} 
 factor through the category of  \'etale Voevodsky motives $D\cM^{\eff}_{\gm, \et} (k;  \bZ)$.  Indeed, the functor  $\hat R^{Hodge}_\bZ$ takes the motive $(\bZ/m\bZ)(n):= cone(\bZ(n)\rar{m} \bZ(n))$ to the complex $\bZ(n)\rar{m} \bZ(n)$ of enriched  Hodge structures of weight $-2n$ and,
 in particular,   $\hat R^{Hodge}_\bZ((\bZ/m\bZ)(n))$ is not isomorphic to $\hat R^{Hodge}_\bZ(\bZ/(m\bZ)(n'))$ in $ D^b(MHS^{\bZ, \en}_{\eff})$, unless $n=n'$.  
 \end{rem}
\subsection{Relation to the work of Gillet and Soul\'e (\cite{gs}) and Bondarko (\cite{bon1}).}\label{s.gs.bon.}  The material of this subsection will not be used in the rest of the paper. 
In (\cite{gs}, Th. 2), Gillet and Soul\'e associated (using algebraic $K$-theory) with any algebraic variety $X$ over a field $k$ of characteristic $0$ an object $W(X)$ in the homotopy category $K^b(\Ho \, \cP)$ of bounded complexes of Chow motives over $k$, called the weight complex of $X$.
In (\cite{bon2}, \S 3.2.1), Bondarko extended Gillet-Soul\'e construction to a triangulated functor\footnote{The weight complex $W(X)$ of Gillet-Soul\'e is obtained by applying this functor to the motive of $X$ with compact support (\cite{voe}, \S 4.1).}
   \begin{equation}\label{gsbwc}
\cG: DM^{\eff}_{\gm} (k;  A)\to K^b(\Ho \, \cP).
  \end{equation}
 His construction can be explained as follows.   Using notation of \S \ref{s.g.c} we have a DG quasi-functor
 \begin{equation}\label{gsbc}
 \cP\liso\cP_{\trunc}\rar{\cG'} \Ho\, \cP,
 \end{equation}
 where the homotopy $\Ho\, \cP$ is viewed as a DG category with the $\Hom$ complexes supported in degree $0$, $\cG'$ is the identity on objects and the projection 
 $$\Hom_{\cP_{\trunc}}(X,Y)=\tau_{\leq 0}\Hom_\cP(X,Y)\to H^0\Hom_\cP(X,Y)=\Hom_{\Ho\, \cP}(X,Y)$$
 on morphisms.   Taking the 
 pretriangulated completion of (\ref{gsbc}) we obtain a DG quasi-functor
    \begin{equation}\label{dggsbwc}
  D\cM^{\eff}_{\gm} (k;  A)\liso \cP^{\pretr}\to (\Ho \, \cP)^{\pretr}=C^b(\Ho \, \cP) 
    \end{equation}
 to the DG category of bounded chain complexes over $\Ho \, \cP$. Passing to the homotopy categories we get  (\ref{gsbwc}). 
 We apply this construction to $k=\bC$.  Following (\cite{gs}, \S 3.1) given a nonnegative  integer $i$
 we consider the functor $\Ho \, \cP  \to HS_{-i}^A$ to the category of pure Hodge structures of weight $-i$ that takes a smooth projective variety $X$ to the Hodge structure on the $i^{th}$ homology group
 $H_i(X,A)$ and denote by $\Gamma_i: C^b( \Ho \, \cP  )\to C^b(HS_{-i}^A)$ the induced DG functor. We set
  $$\Gamma: C^b( \Ho \, \cP  ) \to \bigoplus_{i\geq 0} C^b(HS_{-i}^A), \quad \Gamma = \bigoplus_i \Gamma_i [i].$$
The composition of $\Gamma$ with (\ref{dggsbwc}) determines a quasi-functor 
   \begin{equation}\label{gsbwch}
   D\cM^{\eff}_{\gm} (k;  A) \rar{} \bigoplus_{i\geq 0} C^b(HS_i^A).
   \end{equation}
 Using (\ref{smproj}) it follows that functor (\ref{gsbwch}) is isomorphic to $\on Gr^{\cdot}_W \circ \hat R^{Hodge}_A$.

 \subsection{Lefschetz $(1,1)$ Theorem.}\label{s.del.coh.}  We refer the reader to Theorem \ref{mth} for a generalization of the following result.
  \begin{lm}\label{Le}  For every $M\in  D\cM^{\eff}_{\gm} (\bC;  A)$,    the functor $\hat R^{Hodge}_A$  induces a quasi-isomorphism
 $$ Hom_{D\cM^{\eff}_{\gm} (\bC;  A)}(M, A(1)) \iso Hom_{ D_{dg}^b(MHS^{A, \en}_{\eff})}(\hat R^{Hodge}_A(M),  A(1)).$$
\end{lm}
 \begin{proof} It suffices to prove the result in the case when $M$ is the motive $M(X)$ of a smooth projective variety $X$. Using that $A(1)[2]$ is a direct summand of $M(\bP^1)$ and  
 formula (\ref{smproj2}) we reduce the Lemma  
 to showing that the morphism
\begin{equation}\label{c.tr}
Hom_{D\cM^{\eff}_{\gm} (\bC;  A)}(M(X), A(1)) \to \tau_{\leq 2} Hom_{ D^b_{dg}(MHS^{A})}( R^{Hodge}_A(M(X)), A(1))
 \end{equation} 
 is a quasi-isomorphism. In fact,  the degree $i$ cohomology group of the complex $Hom_{D\cM^{\eff}_{\gm} (\bC;  A)}(M(X), A(1))$ is identified by Voevodsky with
  $H^{i-1}_{\text{Zar}}(X, \cO_{X}^* )\otimes A$ (see, {\it e.g.} \cite{bv}, \S 3.2), which is non-zero only when $i=1$ or $2$. On the other hand,  using Beilinson's computation of $\Ext$ groups in $MHS^{A}$ (\cite{bei}, Lemma 2.3) and the Lefschetz $(1,1)$ Theorem
  the cohomology groups of the complex $Hom_{ D^b_{dg}(MHS^{A})}( R^{Hodge}_A(M(X)), A(1))$ in degrees $i\leq 2$ are also canonically identified with\footnote{
  At the beginning of \S 5.4 Beilinson seems to claim that the cohomology groups of the above complex in degrees $>2$ vanish. This is not true: for example, if $i>3$ the $i^{th}$ cohomology group is isomorphic to
  the Betti cohomology group $H^{i-1}(X, A\otimes\bQ/A)$. In a sense, the reason for introducing the category of enriched Hodge structures is precisely to deal with this problem.}   $H^{i-1}_{\text{Zar}}(X, \cO_{X}^* )\otimes A$.
  Thus, to finish the proof it suffices to check the commutativity of the following diagram (for $i=1,2$).
     \begin{equation}\label{chern}
\def\normalbaselines{\baselineskip20pt
\lineskip3pt  \lineskiplimit3pt}
\def\mapright#1{\smash{
\mathop{\to}\limits^{#1}}}
\def\mapdown#1{\Big\downarrow\rlap
{$\vcenter{\hbox{$\scriptstyle#1$}}$}}
\begin{matrix}
 Hom_{DM^{\eff}_{\gm} (\bC;  A)}(M(X), A(1)[i])  &   \rar{} &Hom_{ D^b(MHS^{A})}( R^{Hodge}_A(M(X)), A(1)[i])   \cr
   \mapdown{\simeq}  & & \mapdown{\simeq}\cr
 H^{i-1}_{\text{Zar}}(X, \cO_{X}^* )\otimes A   &  \rar{Id}  & H^{i-1}_{\text{Zar}}(X, \cO_{X}^* )\otimes A.
\end{matrix}
 \end{equation}
 If $i=1$ the assertion is clear because the functoriality of the isomorphisms of Voevodsky and Beilinson reduces the claim to the case when $X$ is a point. For $i=2$ we observe that because $X$ is projective
 the Picard group  $H^{1}_{\text{Zar}}(X, \cO_{X}^* )$ is generated by classes of very ample line bundles. Thus,  we reduce the statement to the case $X= \bP^n$, where it is immediate.   
  \end{proof}
 \begin{rem} Unless $A=\bQ$,  Lemma \ref{Le} does not hold with  $MHS^{A, \en}_{\eff}$ replaced by $MHS^{A}_{\eff}$.
    \end{rem}
 \subsection{Hodge Albanese functor.}\label{s.one.en.hs.}   
Let $ MHS_1^A $ be 
the full subcategory of  the category $MHS^{A}$ that consists of torsion-free objects of type $\{(0,0),(0,-1),(-1, 0),(-1,-1)\}$. 
We consider the exact structure on $ MHS_1^A $ induced by its embedding  into the exact category $MHS_{\eff}^{A, en}$ and denote by 
  $D^b_{dg}(MHS_1^A)$  the corresponding derived DG category.
 
 
 \begin{pr}\label{halb} The functor  $$D^b_{dg}(MHS^{A}_1)\to D_{dg}^b(MHS^{A, \en}_{\eff})$$ induced by the embedding $MHS_1^A \mono MHS^{A, \en}_{\eff}$  is homotopically fully faithful and has a left 
adjoint DG quasi-functor
 \begin{equation}\label{hadjoint}
\overline{\LAlb}_A:  D_{dg}^b(MHS^{A, \en}_{\eff}) \to D^b_{dg}(MHS^{A}_1).
 \end{equation}
 We call $\overline{\LAlb}_A$  the Hodge Albanese functor. 
\end{pr}
\begin{proof} The first claim follows from the fact that the subcategory $MHS^{A}_1\subset MHS^{A, \en}_{\eff}$ is closed under extensions and from the vanishing of 
$Hom_{D^b(MHS^{A, \en}_{\eff})}(V,V'[i])$,  for $i>1$ and $V,V' \in MHS^{A, \en}_{\eff}$.
For the existence of a left adjoint functor it suffices to check that,  there exists a set $\cS$ of generators of  $D^b(MHS^{A, \en}_{\eff})$ such that, for every $V\in \cS$ the functor 
$$\Phi_V: D^b(MHS^{A}_1) \to Mod_A, \quad \Phi_V(V')=Hom_{D^b(MHS^{A, \en}_{\eff})}(V,V')$$
to the category of $A$-modules
is representable by an object of $D^b(MHS^{A, \en}_{1})$. We take for $\cS$ the set of pure enriched Hodge structures. If $V$ has weight $ <-2$, then  $\Phi_V$ is $0$. Let $V$ be a pure polarizable Hodge structure of weight $-2$. There exists a unique decomposition  $V\otimes _A \bQ $
into the direct sum of a Hodge-Tate substructure and a
substructure $P\subset V\otimes _A\bQ  $ that has no Hodge-Tate
subquotients. Then $\Phi_V$ is representable by $Im(V \to  V\otimes _A\bQ/P)$.
Next, every pure weight $-1$ Hodge structure $V$ is the direct sum of a torsion free Hodge structure $V_f$ and  Hodge structures
of the form $A/mA$, $m\in \bZ$  (with $W_{-1}(A/mA)=A/mA$, $W_{-2}(A/mA)=0$). The functor $\Phi_{V_f}$ is representable by $V_f$ and the functor 
$\Phi_{A/mA}$ is representable by  the complex $A(1)\rar{m} A(1)$ supported in degrees $-1$ and $0$. 
Finally, every effective weight $0$ Hodge structure is already in $ MHS_1^A $.
\end{proof}
 
\section{The Albanese functor}
\subsection{DG category of 1-motives.}\label{s.alb.f.comp.} 
Denote by ${\cM}_1(k)$ the category of
1-motives over a perfect field $k$. Thus, an object of ${\cM}_1(k)$ is a complex of group schemes
                        $$M=[\Lambda \rar{u} G],  $$
      where $\Lambda $ is a $k$-lattice viewed as a group scheme over $k$ 
         and $G$ is a semi-abelian $k$-scheme.  Morphisms between 1-motives are given by commutative squares.
          We set 
$${\cM}_1(k; A): ={\cM}_1(k) \otimes A.$$
For an extension $k\subset k'$, we denote by 
 $$f^*_{k'}:  {\cM}_1(k; A)  \to {\cM}_1(k'; A)$$
 the corresponding pull-back functor.

Every object $M=[\Lambda \rar{u} G]$ of ${\cM}_1(k; A)$ is equipped with a canonical filtration
$$W_{-2} M\subset W_{-1}M \subset W_0M=M, \quad
W_{-2}M= [0\to T], \quad W_{-1} M= [0\to G],
$$
where $T\subset G$ is the toric part $G$.
We refer to $W_iM$ as the {\it weight filtration} on $M$.
 
For an abelian group scheme $H$ over $k$ we denote by   $\underline H$ the corresponding presheaf  of $A$-modules on the category $Sm_k$ of smooth schemes over $k$:
 $$\underline{H}(X): = Hom(X, H)\otimes A.$$
Recall from (\cite{bk}, Lemma 1.3.2) that if the neutral component $H^0$ is quasi-projective then the presheaf
$\underline H $ has a unique structure of a presheaf
with transfers. Using this fact we define a DG functor
     \begin{equation}\label{eqexst}
Tot_A: C(  \cM_1(k, A)) \to  C(PSh^{k, A}_{tr})
 \end{equation}
from the DG category of complexes over $\cM_1(k, A)$ to the DG category of  complexes over the category $PSh^{k, A}_{tr}$ of presheaves with transfers on $Sm_k$ 
sending  a 1-motive $[\Lambda \rar{u} G] \in \cM_1(k, A) \subset  C^b(  \cM_1(k, A))$ to the complex 
$$\underline \Lambda \stackrel{u}{\longrightarrow} \underline G,$$
where  $\underline \Lambda$ is placed in degree $0$ and $\underline G$ in degree $1$.

We introduce an exact structure on $\cM_1(k, A)$ postulating  that a sequence 
\begin{equation}\label{seq}
0\to M^0 \to M^1 \to M^2 \to 0
\end{equation}
 is exact if  the sequence of pure 1-motives  over an algebraic closure $k\subset \ok$   
$$0\to  Gr^i _W f^*_\ok M^0 \to Gr^i _W f^*_\ok M^1 \to Gr^i _W f^*_\ok M^2\to 0, \quad i=0,-1, -2 $$ {\it splits}.
The axioms of exact category (see, {\it e.g.} \cite{p}, Appendix A) are immediate.
\begin{rem}\label{remandisc} A bounded complex of 1-motives $M^{\cdot}\in C^b(  \cM_1(k, A))$ is exact with respect to the above exact structure  if and only if
 $Tot_A(f^*_\ok M^\cdot) \in C^b(PSh^{\ok, A}_{tr})$ is acyclic locally for the {\it Zariski} topology on $Sm_\ok$.
\end{rem}
Denote by  $D^b_{dg}({\cM}_1(k; A))$ the derived DG category  of  the exact category ${\cM}_1(k; A)$ and by
$D^b_{dg}(\w_{\leq i}{\cM}_1(k; A))$, $( i=0,-1, -2)$,  the derived DG category  of  the subcategory $\w_{\leq i}{\cM}_1(k; A) \subset {\cM}_1(k; A)$ consisting of 1-motives whose weights are $\leq i$. 
As the functor $W_i: \cM_1(k, A) \to \w_{\leq i}\cM_1(k, A)$ which takes a 1-motive $M$ to $W_iM$ is exact it determines a quasi-functor
\begin{equation}\label{weightfeq}
W_i: D^b_{dg}({\cM}_1(k; A))\to D^b_{dg}(\w_{\leq i}{\cM}_1(k; A)), \quad  i=0,-1, -2,
\end{equation}
which is right adjoint to the embedding
$$D^b_{dg}(\w_{\leq i}{\cM}_1(k; A))\mono D^b_{dg}({\cM}_1(k; A)).$$
We will write
\begin{equation}\label{weightfgreq}
\on Gr^i_W:  D^b_{dg}({\cM}_1(k; A))\to D^b_{dg}(\w_{= i}{\cM}_1(k; A)), \quad  i=0,-1, -2,
\end{equation}
for the DG quasi-functor to the derived DG category of the subcategory $\w_{= i}{\cM}_1(k; A)\subset {\cM}_1(k; A)$  of pure 1-motives of weight $i$, which carries a complex $M^{\cdot}$ to  $\on Gr^i_W M^{\cdot}$.

\begin{rem} According to Deligne (\cite{d3})  the Hodge realization functor:
$$ \T^{Hodge}_A: {\cM}_1(\bC; A) \iso MHS^A_1$$
is an equivalence of categories.  If we endow $MHS^A_1$ with an exact structure induced by the embedding $MHS_1^A \mono MHS^{A, \en}_{\eff}$ 
functor  $ \T^{Hodge}_A$ is an equivalence of exact categories. In particular, it yields a quasi-equivalence of the corresponding derived DG categories:  
$$ \T^{Hodge}_A: D^b_{dg}({\cM}_1(\bC; A)) \iso D^b_{dg}(MHS^{A}_1).$$
 \end{rem}
\begin{rem}\label{bke} For a perfect field $k$, whose characteristic is either $0$ or invertible in $A$, Barbieri-Viale and Kahn (\cite{bk}, \S 1.4) consider another exact structure on ${\cM}_1(k; A) $ in which sequence (\ref{seq}) is exact if 
the complex $Tot_A( M^{\cdot}) $  is acyclic locally for the \'etale topology on $Sm_k$.  This exact structure has more exact sequences than the one introduced above.
For example,  the sequence
$$0\to [\bZ \to 0] \rar{\alpha}  [\bZ \rar{u} \bG_m]  \rar{\beta} [0\to \bG_m]\to 0,$$
where $u$ takes $1\in \bZ$ to a primitive $n^{th}$ root of unity in $k^*$ and $\alpha$ (resp. $\beta$) acts on $\bZ$ (resp. on $\bG_m$) as the multiplication by $n$, is exact in the sense of Barbieri-Viale and Kahn  but,
unless $1/n \in A$, this sequence is not exact in our sense.  For $k=\bC$   
Barbieri-Viale-Kahn's exact structure is induced by the embedding ${\cM}_1(\bC; A) \iso MHS^A_1\mono   MHS^A$ into the abelian category of mixed Hodge structures.
 \end{rem}
We compute $\Ext$ groups between irreducible objects in  $\cM_1(k, A)$.
\begin{pr}\label{comp} 
Let $M_0=[\Lambda \to 0], M_1= [0 \to Y], M_2=[0\to T] $ be pure 1-motives of weights $0$, $-1$ and $-2$ respectfully. 
 \begin{enumerate}[(a)] 
 \item If $i>j$ then 
   $\Ext^n_{D^b( \cM_1(k, A))}(M_i,M_j)=0$ for every $n$.
 \item  For $i=0,1,-2$, we have 
  $$\Ext^n_{D^b( \cM_1(k, A))}(M_i,M_i)\iso  H^n_{\et}(\spec k,  Hom_{ \cM_1(\ok, A)}(f^*_\ok M_i, f^*_\ok M_i)),$$
  where  $H^*_{\et}(\spec k, ?)$ denotes the Galois cohomology with coefficients in a discrete module. 
 \item We have natural isomorphisms
 $$Ext^n_{D^b( \cM_1(k, A))}(M_0,M_1) \iso H^{n-1}_{\et}(\spec k,  Y (\ok) \otimes \Lambda^*(\ok)\otimes A),$$
$$   Ext^n_{D^b( \cM_1(k, A))}(M_0,M_2) \iso H^{n-1}_{\et}(\spec k, T(\ok) \otimes \Lambda^*(\ok) \otimes A),$$
$$  Ext^n_{D^b( \cM_1(k, A))}(M_0,M_2) \iso H^{n-1}_{\et}(\spec k, Y^\circ (\ok) \otimes \Xi_T(\ok)\otimes A),$$
where $Y^\circ$ is the dual abelian variety and $\Xi_T$ is the $k$-lattice of cocharacters of $T$ ({\it i.e.}, homomorphisms from $\bG_m$ to $T$).
\end{enumerate}
\end{pr}
\begin{proof}
It suffices to prove the proposition for $A=\bZ$. Assume, first, that $k$ is algebraically closed. Then, arguing as in the proof of Lemma \ref{posiclm}, we see that
$\Ext^n_{D^b( \cM_1(k))}(M_i,M_j)=0$ for $n>j-i$.  Thus, it remains to check that $\Ext^2_{D^b( \cM_1(k))}(M_0,M_2)=0$.
Every element of $\Ext^2_{D^b( \cM_1(k))}(M_0,M_2)$ lies in the image of the Yoneda product map 
\begin{equation}\label{yonedaprod}
\Ext^1_{D^b( \cM_1(k))}(M_0,M_1) \otimes \Ext^1_{D^b( \cM_1(k))}(M_1,M_2)\to \Ext^2_{D^b( \cM_1(k))}(M_0,M_2)
\end{equation}
for some pure 1-motive $M_1= [0 \to Y]$ of weight one. Let us show that (\ref{yonedaprod}) is identically zero. Indeed, given $\alpha \in \Ext^1_{D^b( \cM_1(k))}(M_0,M_1) $ and $\beta \in \Ext^1_{D^b( \cM_1(k))}(M_1,M_2)$ we need to check
that there exists a 1-motive $M$ such that  $\on Gr^i _W M= M_{-i}$ and such that the extension classes $[M/W_{-2}M] \in \Ext^1_{D^b( \cM_1(k))}(M_0,M_1)$, $[W_{-1}M] \in \Ext^1_{D^b( \cM_1(k))}(M_1,M_2)$ are equal to $\alpha$ and $\beta$
respectfully. Let $0\to T \to G \to Y \to 0$ be an extension whose class equals $\beta$ and $ [\Lambda \rar{u} Y]$  be an extension whose class equals $\alpha$. Since $k$ is algebraically closed there exists a morphism 
$\tilde u: \Lambda \to G$ that lifts $u$. We take $M= [\Lambda \rar{u} G]$.

In general, for every field $k$ and a finite Galois extension $k\subset k'$ the functor $f^*_{k'}$ has a right adjoint functor $f_{k' *}:  \cM_1(k') \to \cM_1(k)$ 
that carries a 1-motive $[\Lambda \to G]$ to $[R_{k'/k} \Lambda \to R_{k'/k} G]$, where $R_{k'/k}$ is the Weil restriction functor.
Since both $f^*_{k'}$ and $f_{k' *}$ are exact they define a pair of adjoint quasi-functors:
$$f^*_{k'}: D^b_{dg}( \cM_1(k)) \to D^b_{dg}( \cM_1(k')), \quad f_{k' *}:  D^b_{dg}( \cM_1(k')) \to D^b_{dg}( \cM_1(k)).$$
Consider the complex $\Hom_{D^b_{dg}( \cM_1(k))}(M_i,  f_{k' *} f^*_{k'} M_j)$. On the one hand, by the adjunction property, this complex is quasi-isomorphic to the complex  $\Hom_{D^b_{dg}( \cM_1(k'))}(f^*_{k'} M_i,   f^*_{k'} M_j)$.
On the other hand, the action of $\Gal(k'/k)$ on $f_{k' *} f^*_{k'} M_j$ defines a $\Gal(k'/k)$-action on $\Hom_{D^b_{dg}( \cM_1(k))}(M_i,  f_{k' *} f^*_{k'} M_j)$.
We claim that the complex $\Hom_{D^b_{dg}( \cM_1(k))}(M_i,  M_j)$ is quasi-isomorphic to the standard complex that computes the group cohomology of $\Gal(k'/k)$ with coefficients in   $\Hom_{D^b_{dg}( \cM_1(k))}(M_i,  f_{k' *} f^*_{k'} M_j)$.
Indeed, writing $\Omega: = f_{k' *} \circ f^*_{k'}$ and $\Omega^n$ for  the $n$-th iterate of $\Omega$, we form a complex
\begin{equation}\label{triple}
M_j \rar{} \Omega M_j \rar{d_1} \Omega^2 M_j \to  \Omega^3 M_j \to \cdots,
\end{equation}
where the first map $M_j \to \Omega M_j $ is induced by the adjunction unit $Id \rar{\epsilon}\Omega $ and $d_n$ is the alternated sum of the morphisms 
$$ Id\otimes \cdots \otimes Id \otimes \epsilon \otimes Id \otimes \cdots \otimes Id:  \Omega^{m}\circ Id  \circ  \Omega^{n-m} \to \Omega^{m}\circ \Omega  \circ  \Omega^{n-m}.$$ 
Complex (\ref{triple}) is exact  (in fact,  $f^*_{k'}$ of this complex is homotopy contractible by Prop. 8.6.10 from \cite{w}). In particular, $\Hom_{D^b_{dg}( \cM_1(k))}(M_i,  M_j)$ is quasi-isomorphic to the complex $\Hom_{D^b_{dg}( \cM_1(k))}(M_i,  (\Omega^\cdot M_j, d_\cdot))$.
The functor $ f^*_{k'} \circ f_{k' *}$ takes every 1-motive $M$ the direct sum of $|\Gal(k'/k)|$ copies of $M$. This identifies  $\Hom_{D^b_{dg}( \cM_1(k))}(M_i,  \Omega^{n+1} M_j)$ with the complex of $n$-cochains on $\Gal(k'/k)$ with values in
 $\Hom_{D^b_{dg}( \cM_1(k))}(M_i,  \Omega M_j)$ and the total complex $\Hom_{D^b_{dg}( \cM_1(k))}(M_i,  (\Omega^\cdot M_j, d_\cdot))$ with 
the standard complex that computes the group cohomology of $\Gal(k'/k)$ with coefficients in   $\Hom_{D^b_{dg}( \cM_1(k))}(M_i,  \Omega M_j)$.
In particular, for every finite Galois extension $k\subset k'$  we get a spectral sequence converging to $\Ext^*_{D^b( \cM_1(k))}(M_i,M_j)$ whose second term is $$H^p_{\et}(\spec k, \Ext^q_{D^b( \cM_1(k'))}(f^*_{k'} M_i, f^*_{k'} M_j)).$$
For a larger finite Galois extension $k'\subset k^{\prime \prime}$ we have a morphism of the corresponding spectral sequences. Passing to the limit over all finite Galois subextensions  $k\subset k' \subset \ok$ and using that the Galois cohomology 
commutes with filtrant direct limits we get a spectral sequence converging to $\Ext^*_{D^b( \cM_1(k))}(M_i,M_j)$ with 
$$E_2^{pq}= H^p_{\et}(\spec k, \Ext^q_{D^b( \cM_1(\ok))}(f^*_{\ok} M_i, f^*_{\ok} M_j)).$$
Since, in our case, the complex $\Hom_{D^b_{dg}( \cM_1(\ok))}(f^*_{\ok} M_i,   f^*_{\ok} M_j)$ has nontrivial cohomology in a single degree the spectral sequence degenerates in the $E_2$ term and we finish the proof of the proposition.
\end{proof}
\subsection{Galois motives.}\label{s.gm}
For an extension $k\subset k'$ we shall write  
$$f^*_{k'}: D\cM^{\eff}(k;  A) \to D\cM^{\eff}(k';  A)$$
for the base change functor. We define the category
 $D\cM^{\eff}_{\gal} (k;  A)$ of {\it Galois motives} to be the DG quotient of $D\cM^{\eff} (k;  A)$ modulo the subcategory consisting of objects that become contractible in  $D\cM^{\eff} (\ok;  A)$ ({\it i.e.}, equal to $0$ in 
the triangulated category $DM^{\eff} (\ok;  A)$).   We will write $D\cM^{\eff}_{\gm, \gal} (k;  A)$ for the homotopy idempotent completion (\cite{bv}, \S1.6.2) of the image of $D\cM^{\eff}_{\gm} (k;  A)$ in  $D\cM^{\eff}_{\gal} (k;  A)$. 
$\Hom$ complexes in the category  $D\cM^{\eff}_{\gal} (k;  A)$ can be described as follows.
Fix a geometric motive $V \in D\cM^{\eff}_{\gm} (k;  A)$ and an arbitrary motive $V' \in D\cM^{\eff} (k;  A)$. 
Then, the assignment
$$\text{a finite extension}\,  k'  \supset k \quad  \mapsto \quad  \Hom_{D\cM^{\eff} (k';  A)}(f^*_{k'} (V), f^*_{k'} (V'))$$
specifies a complex of presheaves on the small \'etale site $(\spec k)_{\et}$ of $\spec k$. Letting $V'$ vary we get quasi-functor 
$$\tilde \Psi_V: D\cM^{\eff} (k;  A) \to D_{dg}(PSh(k))$$
to the derived DG category of presheaves on $(\spec k)_{\et}$.
Now, since $V$ is compact object of  $D\cM^{\eff}(k;  A)$, we have that
$$H^*\Hom_{D\cM^{\eff}(\ok;  A)}(f^*_{\ok} (V), f^*_{\ok} (V'))\iso \lim_{k\subset k' \subset \ok} H^*\Hom_{D\cM^{\eff} (k';  A)}(f^*_{k'} (V), f^*_{k'} (V')).$$
In particular, if $f^*_{\ok}(V')$ is contractible, then the complex of presheaves $\tilde \Psi_V(V')$ is locally (for the \'etale topology) acyclic {\it i.e.}, the sheafification of   $\tilde \Psi_V(V')$ is an acyclic complex of sheaves on $(\spec k)_{\et}$.
Hence, $\tilde \Psi_V$ determines a quasi-functor 
$$\Psi_V: D\cM^{\eff}_{\gal} (k;  A) \to D_{dg}(Sh_{\et}(k))$$
from the category of Galois motives to the derived DG category of sheaves on $(\spec k)_{\et}$.
Observe that $\Psi_V(V)$ has a canonical global  section
$$1\in H^0\Hom_{ D_{dg}(Sh_{\et}(k)) }(\bZ, \Psi_V(V))= R^0\Gamma_{\et}(\spec k, \Psi_V(V)) $$
corresponding to the identity morphism $V\to V$ in $D\cM^{\eff} (k;  A)$.
If $V' \in D\cM^{\eff}_{\gal} (k;  A)$ is a Galois motive, then
the morphism 
$$\Hom_{D\cM^{\eff}_{\gal}(k;  A)}(V, V')\to \Hom_{ D_{dg}(Sh_{\et}(k)) }(\Psi_V(V), \Psi_V(V'))$$
together with the  above global section of  $\Psi_V(V)$ specify a morphism 
\begin{equation}\label{nicecomp}
\Hom_{D\cM^{\eff}_{\gal}(k;  A)}(V, V')\to R\Gamma_{\et}((\spec k, \Psi_V(V')).
\end{equation}
\begin{pr}\label{compgm}
For every $V\in D\cM^{\eff}_{\gm, \gal} (k;  A)$ and  $V'\in D\cM^{\eff}_{\gal} (k;  A)$ morphism (\ref{nicecomp}) is a quasi-isomorphism.
\end{pr}
\begin{proof}
Consider the functor
$$\tilde \Phi_V: D\cM^{\eff}_{\gal} (k;  A) \to D_{dg}(PSh(k)),$$
which carries a Galois motive $V'$ to the complex
 $$\text{a finite extension}\,  k'  \supset k \quad  \mapsto \quad  \Hom_{D\cM^{\eff}_{\gal} (k';  A)}(f^*_{k'} (V), f^*_{k'} (V'))$$
For every $V'$, the complex $\tilde \Phi_V(V')$ is \'etale local {\it i.e.}, it satisfies the decent property for  \'etale hypercoverings (\cite{bv}, 1.11).
In addition, we have a morphism $\tilde \Psi_V(V') \to \tilde \Phi_V(V')$ in $D_{dg}(PSh(k))$ that induces a quasi-isomorphism of the corresponding complexes of sheaves:
the fibers of  $\Psi_V(V')$ and $\Phi_V(V')$  at the geometric point $\spec \ok \to \spec k$
are quasi-isomorphic to $\Hom_{D\cM^{\eff} (\ok;  A)}(f^*_{\ok} (V), f^*_{\ok} (V'))$. Using (\cite{bv}, 1.11), it follows that
$$\Hom_{ D_{dg}(Sh_{\et}(k)) }(\bZ, \Psi_V(V))\iso \Hom_{ D_{dg}(Sh_{\et}(k)) }(\bZ, \Phi_V(V))\iso \Hom_{ D_{dg}(Sh_{\et}(k)) }(\bZ, \Psi_V(V')).$$
\end{proof}
We abbreviate the above computation in the formula
\begin{equation}
\Hom_{D\cM^{\eff}_{\gal}(k;  A)}(V, V')\iso R\Gamma_{\et}(\spec k,  Hom_{D\cM^{\eff} (\ok;  A)}(f^*_{\ok} (V), f^*_{\ok} (V'))).
\end{equation}

\subsection{Motivic Albanese functor.}\label{s.embedd}
Using Remark \ref{remandisc}, the functor $Tot_A: C^b(\cM_1(k; A)) \to C^b(PSh^{k, A}_{\tr})$ 
 determines a quasi-functor
\begin{equation}\label{bkemm}
\iota: D^b_{dg}({\cM}_1(k; A)) \to D\cM^{\eff}_{\gal} (k;  A).
\end{equation}
\begin{rem} We conjecture that $\Ho\, \iota$ is fully faithful. We do not know if $\iota$ factors through the subcategory $D\cM^{\eff}_{\gm, \gal} (k;  A)\mono D\cM^{\eff}_{\gal} (k;  A)$ of geometric Galois motives.
\end{rem}

We shall prove that $\iota$ has a left adjoint quasi-functor
$$\LAlb_A: D\cM^{\eff}_{\gm, \gal} (k;  A) \to D^b_{dg}({\cM}_1(k; A))$$
{\it i.e.}, a quasi-functor  together a morphism 
$$\nu: Id \to \iota \circ \LAlb_A: D\cM^{\eff}_{\gal, \gm} (k;  A) \to D\cM^{\eff}_{\gal} (k;  A)$$
such that, for every $V\in  D\cM^{\eff}_{\gm, \gal} (k;  A)$ and $M\in D^b_{dg}({\cM}_1(k; A))$, the composition 
$$H^*\Hom_{D^b_{dg}({\cM}_1(k; A))}(\LAlb_A(V), M) \to  H^*\Hom_{ D\cM^{\eff}_{\gal} (k;  A) }(\iota \circ \LAlb_A(V),  \iota(M)) $$
$$\rar{\nu^*} H^*\Hom_{D\cM^{\eff}_{\gal} (k;  A)}(V, \iota(M)) $$
is an isomorphism.

Let  
\begin{equation}\label{foradjointf}
\iota^*:  \underleftarrow{D\cM^{\eff}_{\gal} (k;  A)}\to \underleftarrow{D^b_{dg}({\cM}_1(k; A))}
\end{equation}
 be the restriction quasi-functor between the corresponding categories of left DG modules (see \S \ref{s.dgcat}). We shall use the same notation
\begin{equation}\label{foradjointfgm}
\iota^*:  D\cM^{\eff}_{\gm, \gal} (k;  A) \to \underleftarrow{D^b_{dg}({\cM}_1(k; A))}
\end{equation}
for the composition of (\ref{foradjointf})  with Yoneda embedding $D\cM^{\eff}_{\gm, \gal} (k;  A)  \to \underleftarrow {D\cM^{\eff}_{\gal} (k;  A)}$. 
Thus, for an object $V\in  D\cM^{\eff}_{\gm, \gal} (k;  A)$ the DG module $\iota^*(V) \in \underleftarrow{D^b_{dg}({\cM}_1(k; A))}$ carries an object $M\in D^b_{dg}({\cM}_1(k; A))$ to 
$$Hom_{D\cM^{\eff}_{\gal} (k;  A)}(V, \iota(M)).$$ We refer to $\iota^*$ as the formal left adjoint functor
of $\iota$.  The next theorem asserts that (\ref{foradjointfgm}) factors through $D^b_{dg}({\cM}_1(k; A))$ and, thus, defines $\LAlb_A$.
\begin{Th}\label{th2} Assume that either $\Char k=0$ or that $\Char k$ is invertible in $A$.  Then the following assertions are true.
 \begin{enumerate}[(a)] 
\item
For every motive $V\in  D\cM^{\eff}_{\gm, \gal} (k;  A)$, the DG module $\iota^*(V) $ is quasi-representable {\it i.e.},  $\iota^*(V) $ is homotopy equivalent to an object lying in the image of the Yoneda embedding 
$D^b_{dg}({\cM}_1(k; A))  \to \underleftarrow {D^b_{dg}({\cM}_1(k; A))}$. 
\item The functor $\LAlb_A$ commutes with an arbitrary base field change $k\subset k'$:
$$ \LAlb_A\circ f_{k'}\iso f_{k'}\circ \LAlb_A.$$
\end{enumerate}
\end{Th}
\begin{proof} Working with DG categories (as oppose to merely triangulated ones) makes the proof easier: 
the subcategory of $D\cM^{eff}_{\gm, \gal} (k;  A)$ consisting of all objects $V$ for which $\iota^*(V) $ is quasi-representable is pretriangulated and homotopy idempotent complete (because $D^b_{dg}({\cM}_1(k; A))$ has these properties).
 Since the triangulated category  $D\cM^{\eff}_{\gm, \gal} (k;  A)$ coincides with its smallest idempotent complete triangulated subcategory that contains  motives of all smooth projective varieties (\cite{bon2}),
 it suffices to prove that, for every smooth connected projective variety $X$, the DG module  $\iota^*(M(X))$ is quasi-representable. 
 
Let $\pi_0(X)$ be a $0$-dimensional scheme, which is the spectrum of  the integral closure of $k$ in $\cO(X)$ and let $X \to \pi_0(X)$ be the canonical morphism.
The DG module  $\iota^*(M(\pi_0(X)))$ is represented by the lattice $\bZ[\pi_0(X)(\ok)]$ spanned by $\ok$-points   of $\pi_0(X)$. Thus, it is enough to prove representability of $\iota^*(\overline M(X))$, where
$$\overline M(X)= \cone(M(X)\to M(\pi_0(X))[-1].$$
We notice that the DG module $\iota^*(\overline M(X))$ carries every object $\Lambda ^{\cdot} \in D^b_{dg}({\cM}_1(k; A))$ of weight $0$ to an acyclic complex {\it i.e.},  $\iota^*(\overline M(X))$ lies in the image of the embedding
$$ \underleftarrow{D^b_{dg}(\w_{\leq -1}{\cM}_1(k; A))}\mono  \underleftarrow{D^b_{dg}({\cM}_1(k; A))}.$$

Let ${\bf Alb} (X)$ be the Albanse variety of $X$. This is an abelian variety over $k$ characterized by the following universal property:  for every abelian variety $Y$ over $k$, we have a functorial isomorphism of $\Gal(\ok/k)$-modules 
\begin{equation}\label{clalb}
\Hom({\bf Alb} (X)_{\ok}, Y_{\ok})\iso \Mor(X_{\ok}, Y_{\ok})/\Mor(\pi_0(X)_\ok, Y_\ok),
\end{equation}
where $\Hom$ denotes the group of morphisms in the category of group schemes over $\ok$ and $\Mor$ denotes the set of morphisms in the category of schemes over $\ok$ with the group structure induced by the group structure on $Y$. 
In particular, $Id\in  \Hom({\bf Alb}(X) , {\bf Alb}(X))$ defines a Galois invariant element of the quotient $$\Mor(X_\ok, {\bf Alb}(X)_\ok)/\Mor(\pi_0(X)_\ok, {\bf Alb}(X)_\ok),$$
which determines  a  morphism
$$\overline M(X) \to \iota([0\to {\bf Alb}(X)])[1].$$
in the category of Galois motives.
By adjunction we get a morphism of DG modules 
$$\alpha: \iota^*(\overline M(X)) \to [0\to {\bf Alb}(X)][1],$$
where we identify an object of $D^b_{dg}({\cM}_1(k; A))$ with its image under the Yoneda embedding.
It suffices to check that $\cone \alpha$ is quasi-representable. Using Propositions \ref{comp}  and \ref{compgm} together with the universal property (\ref{clalb}) it follows that $\cone \alpha$ carries every object 
$M^{\cdot} \in D^b_{dg}({\cM}_1(k; A))$ whose weights are $> -2$ to $0$  {\it i.e.},
$\cone \alpha$ lies in the image of the embedding
$$ \underleftarrow{D^b_{dg}(w_{\leq -2}{\cM}_1(k; A))}\mono  \underleftarrow{D^b_{dg}({\cM}_1(k; A))}.$$

We shall construct a complex  $NS_{X_{\ok}}^*(1)[3]\in D^b_{dg}(w_{\leq -2}{\cM}_1(k; A))$ that represents $\cone \alpha$ explicitly. Let $NS_{X_{\ok}}$ be the N\'eron-Severi  group of   $X_\ok$. This is a finitely generated abelian group
equipped with an action of $\Gal(\ok/k)$. Write $NS_{X_{\ok}}$ as the cokernel of two $k$-lattices:
$$0\to \Lambda_1\to  \Lambda_2 \to NS_{X_{\ok}}\to 0$$
and let $NS_{X_{\ok}}^*(1)$ be the complex of 1-motives
 $$[0\to  \Lambda^*_2(1)] \to [0\to  \Lambda^*_1(1)]$$ 
where  $\Lambda^*_i(1)$ is the torus whose characters lattice is  $\Lambda_i$. Then,  $NS_{X_{\ok}}^*(1)$ viewed as an object of  $D^b_{dg}(w_{\leq -2}{\cM}_1(k; A))$ does not depend on the choice of the presentation because
for any such choice $NS_{X_{\ok}}^*(1)$ represents the following functor:
$$\Hom_{D^b_{dg}({\cM}_1(k; A))}(NS_{X_{\ok}}^*(1), \Xi(1))\iso R\Gamma_{\et}(\spec k, NS_{X_{\ok}}\otimes \Xi) \otimes A,$$
where $\Xi$ is a $k$-lattice.  
On the other hand, using Propositions \ref{compgm} and  \ref{comp},
it follows that the DG module $\cone \alpha$ carries 1-motive $\Xi(1)$ to 
$$R\Gamma_{\et}(\spec k,   \cone(Pic^0({\bf Alb}(X_\ok))\to  Pic(X_\ok))\otimes \Xi)\otimes A[-3].$$
Since $Pic^0({\bf Alb}(X_\ok))\iso  Pic^0(X_\ok)$, we get an isomorphism $\cone \alpha \iso NS_{X_{\ok}}^*(1)[3]$, which completes the proof of part (a) of the Theorem.

\begin{rem}\label{albsmpr} The above proof shows that for a smooth projective variety $X$ we have a functorial quasi-isomorphisms
$$\on Gr^{-2}_W \LAlb_A(X)\iso  NS_{X_{\ok}}^*(1)[2],$$
$$\on Gr^{-1}_W \LAlb_A(X)\iso  [0\to {\bf Alb}(X)][1],$$
$$\on Gr^{0}_W \LAlb_A(X)\iso  [\bZ[\pi_0(X)(\ok)] \to 0].$$
\end{rem}
\begin{rem}\label{integralincharp}  Let $\cP \subset  D\cM^{\eff}_{\gm, \gal} (k;  A)$ be the full
 DG subcategory consisting of those motives that are isomorphic in the homotopy category $DM^{\eff}_{\gm, \gal} (k;  A)$ to direct summands of motives of smooth projective varieties and let $\cP^{\pretr}\mono  D\cM^{\eff}_{\gm, \gal} (k;  A) $ be
its pretriangulated completion. Then, even   if $\Char p$ is not invertible in $A$,  the above proof shows, for every $V\in  \cP^{\pretr}$, the functor $\iota^*(V) $ is quasi-representable. This defines a quasi-functor
$$ \LAlb_A: \cP^{\pretr} \to D^b_{dg}({\cM}_1(k; A)).$$
\end{rem}

To prove part (b) of the Theorem recall from (\cite{vol}, Prop. 3.3) that the quasi-functor $\iota$ commutes with the an arbitrary base change $k\subset k'$. Thus, by adjunction, we get a morphism of quasi-functors
$$ \LAlb_A\circ f_{k'}\to f_{k'}\circ \LAlb_A.$$
To prove that this morphism is an isomorphism it suffices to show that for every smooth projective variety $X$ over $k$ the induced morphism
 $$\on \Gr^{i}_W  \LAlb_A(M(X_{k'}))\to f_{k'}(\on Gr^{i}_W  \LAlb_A(M(X))),$$
 which follows from Remark (\ref{albsmpr}).

\end{proof}

\subsection{The Deligne conjecture.}\label{s.deligne.conj.} 
 \begin{Th}\label{mth} For $k\subset \bC$, 
we have a commutative diagram of quasi-functors
\begin{equation}\label{diagmaindc}
\def\normalbaselines{\baselineskip20pt
\lineskip3pt  \lineskiplimit3pt}
\def\mapright#1{\smash{
\mathop{\to}\limits^{#1}}}
\def\mapdown#1{\Big\downarrow\rlap
{$\vcenter{\hbox{$\scriptstyle#1$}}$}}
\begin{matrix}
D\cM^{\eff}_{\gm}( k ; A)  & \rar{\LAlb_A}  &D^b_{dg}(\cM_1(k, A))  \cr
  \mapdown{\hat R^{Hodge}_A } &   &\mapdown{\T^{Hodge}_A }  \cr
 D^b_{dg}(MHS_{\eff}^{A, \en})   &\rar{\overline{\LAlb}_A } &  D^b_{dg}( MHS_1^{A,\en}).
\end{matrix}
 \end{equation}

\end{Th}
\begin{proof}
We shall start by recalling a general Homological Algebra construction.
Let $\cR$ be a negative DG category over a commutative ring $A$ (see \S \ref{s.trunc}),
and let $\cT(\cR, D_{dg}(Mod_A))$ be the triangulated category of left DG modules  over $\cR$ (\S \ref{s.dgcat}).
Denote by $\cT(\cR, D_{dg}(Mod_A))^{\leq 0}$ (resp.  $\cT(\cR, D_{dg}(Mod_A))^{\geq 0}$)   
the full subcategory of $\cT(\cR, D_{dg}(Mod_A))$ formed by those DG functors $L$ that carry every object $X\in \cR$ to a complex $L(X)$  acyclic in positive degrees (resp. in negative degrees).
Then, by Remark \ref{tstructurebis}  the subcategories 
$$\cT(\cR, D_{dg}(Mod_A))^{\leq 0}, \cT(\cR, D_{dg}(Mod_A))^{\geq 0} \subset \cT(\cR, D_{dg}(Mod_A))$$ 
define a $t$-structure on $\cT(\cR, D_{dg}(Mod_A))$. In particular,  the embedding $\cT(\cR, D_{dg}(Mod_A))^{\leq 0} \mono \cT(\cR, D_{dg}(Mod_A))$
admits a right adjoint functor
 \begin{equation}\label{gen3}
\tau_{\leq 0}: \cT(\cR, D_{dg}(Mod_A)) \to \cT(\cR, D_{dg}(Mod_A))^{\leq 0}. 
 \end{equation}

 Now we come back to the proof of Theorem \ref{mth}.
 Since the motivic Albanese functor commutes with the base change, it suffices to construct diagram (\ref{diagmaindc}) for $k=\bC$.
 We need to show that the following two left DG modules over  $D\cM^{\eff}_{\gm}(\bC ; A)^{op} \otimes ^L  D^b_{dg}({\cM}_1(\bC; A))$ 
   \begin{equation}\label{gen1}
          V\otimes M \quad \mapsto \quad  \Hom_{D\cM^{\eff}(\bC ; A)}(V, \iota(M))
 \end{equation}          
   \begin{equation}\label{gen2}
          V\otimes M \quad \mapsto \quad  \Hom_{D^b_{dg}( MHS^{A,\en})}(\hat R^{Hodge}_A(V), \T^{Hodge}_A(M))
 \end{equation}   
  are quasi-isomorphic.
  Consider the full subcategory $\cQ$ of $D^b_{dg}({\cM}_1(\bC; A))$ formed by those objects
that are isomorphic in the triangulated category $D^b({\cM}_1(\bC; A))$ to an object of the form $M[i]$, $i=0,-1, -2$, where $M$ is  a pure 1-motive of weight $-i$.  The category $D^b_{dg}({\cM}_1(\bC; A))$ is generated by $\cQ$ and, using  
Proposition \ref{comp}, it follows that $\cQ$ is negative. Let $\cP\subset D\cM^{\eff}_{\gm}(\bC ; A)$ be the Chow weight structure, and
let $L_1$ and $L_2$ be the restriction of DG modules  (\ref{gen1}) and  (\ref{gen2}) to $\cR: = \cP^{op}\otimes^L \cQ$. 
Then, since the embedding  
$$\cR \mono D\cM^{\eff}_{\gm}(\bC ; A)^{op} \otimes^L  D^b_{dg}({\cM}_1(\bC; A))$$
yields a quasi-equivalence of the corresponding pretriangulated  completions, it suffices to prove that  $L_1$ and $L_2$ are quasi-isomorphic.
Consider a third DG module over $\cR$
 \begin{equation}\label{gen4}
 V\otimes M \quad \mapsto \quad  \Hom_{D^b_{dg}( MHS^{A})}(R^{Hodge}_A(V), \T^{Hodge}_A(M)),
    \end{equation}        
   where we write, by abuse of notation, $\T^{Hodge}_A(M)$ both for an object of  $D^b_{dg}( MHS^{A, \en})$ and for its image in the derived
   DG category $D^b_{dg}( MHS^{A})$ of Deligne's mixed Hodge structures.
   According to (\cite{vol}, Theorem 3) we have an isomorphism of DG functors
   $$R^{Hodge}_A \circ \iota \iso \T^{Hodge}_A: D^b_{dg}({\cM}_1(\bC; A)) \to D^b_{dg}( MHS^{A}).$$
   Using this isomorphism we get a morphism of DG modules
   $$v: L_1\to L_3 $$
   $$ L_1(V\otimes M) \to \Hom_{D^b( MHS^{A})}(R^{Hodge}_A(V), R^{Hodge}_A \circ \iota (M))\iso L_3(V\otimes M).$$
   We also have a morphism
   $$u: L_2 \to L_3$$
   induced by the functor $D^b( MHS^{A, \en})\to D^b( MHS^{A})$.
   Next,  we apply the canonical truncation (\ref{gen3}) to the DG modules $L_i$ over the negative category  $\cR$.   Consider morphisms    
     \begin{equation}\label{gen5}
     \tau_{\leq 0}v: \tau_{\leq 0}L_1\to \tau_{\leq 0}L_3 
     \end{equation}
     $$\tau_{\leq 0}u: \tau_{\leq 0}L_2\to \tau_{\leq 0}L_3 $$
  of truncated DG modules induced by  $v$ and $u$.
We claim that $\tau_{\leq 0}u$ as well as the canonical morphisms $\tau_{\leq 0}L_2\to L_2$ and $\tau_{\leq 0}L_1\to L_1$ are quasi-isomorphisms.
  The first two assertions follow from the fact that, for $V\in \cP$ and $M\in \cQ$, both $\T^{Hodge}_A(M), \hat R^{Hodge}_A(V) \in D^b( MHS^{A, \en})$ are isomorphic  
   to objects of the form $\oplus_i M_i[i]$,  where $M_i$ is  a pure Hodge structure of weight $-i$, and  Lemma \ref{posiclm}. The last assertion is trivial. Finally, we have to check that $\tau_{\leq 0}v$ is an isomorphism.
   This is obvious when $V$ is the motive of a point. Thus, it suffices to prove the statement for the reduced motives of smooth projective varieties $\overline M(X)= \cone(M(X)\to M(spec \, \bC))[-1]$. If $M$ is a 1-motive  of weight $0$ both sides of 
 (\ref{gen5}). For $M=[0\to Y][1]$, where $Y$ is an abelian variety, 
 the assertion amounts the classical statement that the sequence
 $$0\to \Mor(spec, \bC, Y)\otimes A \to \Mor(X, Y)\otimes A \to \Hom_{MHS^A}(H_1(X, \bZ), H_1(Y, \bZ))\to 0$$
  is exact. Finally, for $M=[0\to \bG_m][2]$ the assertion is proven in formula (\ref{c.tr})  from \S \ref{s.del.coh.}. 
\end{proof}
\begin{rem}\label{clasdelconj} Let us explain how the assertion of Theorem \ref{mth} is related to the ``classical'' Deligne conjecture on 1-motives, proven in \cite{brs} and in \cite{r}. 
This conjecture asserts that, for every variety $X$ over $k\subset \bC$ and an
integer $m\geq 0$, there exists 1-motive $L_mAlb_\bZ(X)$ whose Hodge realization is isomorphic to the maximal quotient of the mixed Hodge structure $H_m(X_\bC, \bZ)$ which belongs to $MHS_1^{\bZ}$. 
Consider the complex $LAlb_\bZ(M(X)) \in D^b_{dg}({\cM}_1(k; \bZ))$. Although the category ${\cM}_1(k; \bZ)$ is not abelian, every morphism in this category has both a kernel and a cokernel. In particular, for every complex of 1-motives $M^\cdot$
one can define its cohomology $H^*(M^\cdot)\in MHS_1^{\bZ}$ by the usual formula\footnote{Warning: $H^0: D^b({\cM}_1(k; \bZ)) \to {\cM}_1(k; \bZ)$ is not a cohomological functor (\cite{w}, Def. 10.2.7).}. 
Using  Theorem \ref{mth}, it follows that the 1-motive
  $$L_mAlb_\bZ(X):= H^{-m} (LAlb_\bZ(M(X)))$$ 
satisfies the required
property. 
\end{rem}

\end{document}